\documentclass[12pt, reqno]{amsart}

\usepackage[utf8]{inputenc}
\usepackage[T1]{fontenc}
\usepackage{lmodern}
\usepackage[final]{microtype}

\usepackage{
    graphicx,
    fancyhdr, 
    amscd,
    mathrsfs,
    amssymb,
    amsmath,
    tikz, 
    tikz-cd, 
    mathtools,
    xfrac,
    overpic,
    stmaryrd
}

\usepackage[shortlabels]{enumitem}
\setlist[enumerate]{leftmargin=*}

\usepackage[%
    colorlinks=true,
    linkcolor=blue,
]{hyperref}

\usepackage[
    backend=biber,
    style=numeric-comp,
    url=false,
    giveninits=true,
]{biblatex}
\usetikzlibrary{calc}

\usetikzlibrary{decorations.pathmorphing}

\tikzset{curve/.style={settings={#1},to path={(\tikztostart)
    .. controls ($(\tikztostart)!\pv{pos}!(\tikztotarget)!\pv{height}!270:(\tikztotarget)$)
    and ($(\tikztostart)!1-\pv{pos}!(\tikztotarget)!\pv{height}!270:(\tikztotarget)$)
    .. (\tikztotarget)\tikztonodes}},
    settings/.code={\tikzset{quiver/.cd,#1}
        \def\pv##1{\pgfkeysvalueof{/tikz/quiver/##1}}},
    quiver/.cd,pos/.initial=0.35,height/.initial=0}

\tikzset{tail reversed/.code={\pgfsetarrowsstart{tikzcd to}}}
\tikzset{2tail/.code={\pgfsetarrowsstart{Implies[reversed]}}}
\tikzset{2tail reversed/.code={\pgfsetarrowsstart{Implies}}}

\tikzset{no body/.style={/tikz/dash pattern=on 0 off 1mm}}

\title{On Realisability of Twisted Homology}
\author{Mark Grant}
\address{Institute of Mathematics, University of Aberdeen, United Kingdom}
\email{mark.grant@abdn.ac.uk}

\author{Michael Jung}
\address{Department of Mathematics, Vrije Universiteit Amsterdam, Netherlands}
\email{m.jung@vu.nl}

\author{Baylee Schutte}
\address{Institute of Mathematics, Freie Universität Berlin, Germany}
\email{bschutte@zedat.fu-berlin.de}


\usepackage[margin=2.5cm,head=.25cm,foot=1.5cm]{geometry}

\numberwithin{equation}{section}

\let\oldtocsection=\tocsection
\let\oldtocsubsection=\tocsubsection
\let\oldtocsubsubsection=\tocsubsubsection
\renewcommand{\tocsection}[2]{\hspace{0em}\oldtocsection{#1}{#2}}
\renewcommand{\tocsubsection}[2]{\hspace{1em}\oldtocsubsection{#1}{#2}}
\renewcommand{\tocsubsubsection}[2]{\hspace{2em}\oldtocsubsubsection{#1}{#2}}

\usepackage{needspace}

\newtheorem{theorem}{Theorem}[section]
\newtheorem{prop}[theorem]{Proposition}
\newtheorem{lemma}[theorem]{Lemma}
\newtheorem{corollary}[theorem]{Corollary}
\newtheorem{question}[theorem]{Question}

\theoremstyle{definition}
\newtheorem{construction}[theorem]{Construction}
\newtheorem{notate}{Notation}
\newtheorem{definition}[theorem]{Definition}
\newtheorem{example}[theorem]{Example}
\newtheorem{remark}[theorem]{Remark}

\def\tw{\mathrm{tw}}

\def\ZZ{\mathbb{Z}}                     
\def\RR{\mathbb{R}}                     
\def\LL{\mathbb{L}}                     
\def\BB{\mathcal{B}}                    
\def\AA{\mathcal{A}}                    
\def\Or{\mathcal{O}}  			        
\def\bwedge{\bigwedge\nolimits}         
\def\B{\mathrm{B}}                      
\def\E{\mathrm{E}}                      

\def\placeholder{\:\!-\:\!}             

\DeclareMathOperator{\pr}{pr}           
\DeclareMathOperator{\Hom}{Hom}         

\DeclareMathOperator{\id}{id}           

\DeclareMathOperator{\SO}{SO}           
\DeclareMathOperator{\St}{St}           
\DeclareMathOperator{\OO}{O}            
\DeclareMathOperator{\UU}{U}            
\DeclareMathOperator{\MSO}{MSO}         
\DeclareMathOperator{\BO}{BO}           
\DeclareMathOperator{\EO}{EO}           
\DeclareMathOperator{\BSO}{BSO}         

\DeclareMathOperator{\Aut}{Aut}         
\DeclareMathOperator{\Gr}{Gr}           
\newcommand{\colim}{\mathop{\mathrm{colim}}\displaylimits}    

\DeclareMathOperator{\twThom}{M^\tw O}    

\newcommand\restr[2]{\ensuremath{\left.#1\right|_{#2}}}     
\newcommand\catname[1]{\mathsf{#1}}                         

\definecolor{vblue}{HTML}{0089cf}
\definecolor{vorange}{HTML}{CF4600}
\definecolor{vgreen}{HTML}{1CC900}
\definecolor{vpurple}{HTML}{8E44AD}

\usepackage{etoolbox}
\patchcmd{\colon}{{:}}{\mathopen:}{}{}

\begin{document}

\begin{abstract}
    We discuss the question of when a homology or cohomology class with twisted integer coefficients of a manifold $X$ is realised by a submanifold.
    While this question is classical in nature, providing an answer requires relatively modern techniques from parametrised homotopy theory.
    More specifically, we introduce cobordism classes twisted by a coefficient system and then define a twisted Thom space $\twThom(n)$ over $\BO(1)$, which serves as the classifying object for this cobordism theory under a twisted Pontryagin--Thom construction.
    As a result, a twisted homology class is realisable if and only if its Poincar\'e dual is the image of the twisted Thom class in $\twThom(n)$ under a parametrised map $X \to \twThom(n)$ over $\BO(1)$.
    By showing that $\twThom(n)$ is parametrised homotopy equivalent to the unbased homotopy orbits of $\MSO(n)$ under the orientation-reversing involution, we are able to construct its parametrised Postnikov tower to derive obstructions to realisability. We conclude by giving the first known examples of non-realisable integer homology classes in non-orientable manifolds.
\end{abstract}

\maketitle

\tableofcontents

\section{Introduction}
This paper addresses one of the oldest questions in algebraic topology, which has shaped the course of this field throughout the last century: when is a homology class in a manifold realised by a submanifold?
Ren\'e Thom famously gave necessary, and in some cases sufficient, conditions in the case of integral and mod\,$2$ homology and used these to solve Steenrod's problem on realising homology classes in finite polyhedra by fundamental classes of manifolds~\cite{Thom54}. Thom's method was to use Poincar\'e duality and what is now called the Pontryagin--Thom construction to convert the question into one of homotopy theory: a homology class is realisable by a submanifold if and only if its Poincar\'e dual cohomology class is induced from the Thom class by some map into the universal Thom space. 

In the integral case, Thom's techniques require that the ambient manifold is orientable. For if $a\in H_k(X;\ZZ)$ is an integral homology class in a \emph{non-orientable} $(n+k)$-manifold, the Poincar\'e dual cohomology class $\alpha\in H^n(X;\Or_X)$ has twisted coefficients in the orientation local system of $X$. Thus the seemingly classical question of when an integral homology class in a non-orientable manifold is represented by an oriented submanifold requires analogues of Thom's results with twisted coefficients.
This question was raised recently by Zhenhua Liu~\cite{LiuMO} and has not been addressed in the literature.

In this paper, we investigate the more general question of when a homology or cohomology class with twisted integer coefficients is realised by a submanifold.
In doing so, we utilise modern tools from parametrised homotopy theory.
Indeed, one might peg this article as further evidence of the recent ``renaissance'' of bordism questions and computations in the twisted and parametrised context, see e.g.~\cite{AyalaFrancis25, CalmesEtAl25}.
Also of note is the parametrised Spin bordism spectrum defined by Hebestreit--Joachim~\cite{HebestreitJoachim20}.
While the theory developed in the present article is similar in spirit, their Spin bordism theory does not directly recover our version of cobordism.
In fact, our notion of cobordism was motivated by a recent application \cite{JungRot25} of the foundational work of Atiyah~\cite{Atiyah61}. Moreover, our arguments for the Pontryagin--Thom construction are rather geometric and inspired by the work of Cruickshank~\cite{Cruickshank03}.

\subsection{Main results}

Let $X$ be a closed connected $(n+k)$-manifold, and let $\iota\colon M \hookrightarrow X$ be an embedding of a closed $k$-dimensional submanifold. Both $M$ and $X$ have canonical twisted fundamental classes, denoted $[M]_\tw \in H_k(M;\Or_M)$ and $[X]_\tw\in H_{n+k}(X;\Or_X)$.
Given a local system of integer coefficients $\AA$ on $X$, we say that a homology class $a\in H_k(X;\AA \otimes \Or_X)$ is \emph{realisable} if there exists an embedding $\iota\colon M\hookrightarrow X$ as above such that
\[
    \iota^*(\AA \otimes \Or_X) \cong \Or_M \quad \text{and} \quad \iota_*[M]_\tw =a \in H_k(X;\AA \otimes \Or_X).
\]
Similarly, we say that a cohomology class $\alpha\in H^n(X;\AA)$ is \emph{realisable} if the twisted Poincar\'e dual $a \coloneq \alpha\cap[X]_\tw\in H_k(X;\AA\otimes\Or_X)$ is realisable.

In Section \ref{sec:twisted_Thom_theory} below we define a \emph{twisted Thom space} $\twThom(n)$, which comes equipped with a sectioned fibration $q\colon \twThom(n)\to \BO(1)$ whose fibre is $\MSO(n)$, turning it into a retractive space over $\BO(1)$.
Note that the twisted Thom space $\twThom(n)$ introduced in this manuscript is different from the typical fibrewise Thom space discussed e.g.\ in \cite[Example~2.1.3]{Malkiewich23}.
Instead of taking the fibrewise quotient of the universal $n$-plane disc bundle over the entire base $\BO(n)$, we collapse its boundary to $\BO(1)$ along the map $w_1 \colon \BO(n) \to \BO(1)$ induced by the determinant. (This could also be described as applying the shriek functor $(w_1)_!$ from retractive spaces over $\BO(n)$ to retractive spaces over $\BO(1)$ to the usual fibrewise Thom space.) Under this construction, the universal local system of integer coefficients $\mathcal{L}_\ZZ$ on $\BO(1) \simeq \operatorname{BAut}(\ZZ)$ pulls back to a local system $\mathcal{L}_\ZZ^{\mathrm{Th}} \coloneq q^*\mathcal{L}_\ZZ$ on $\twThom(n)$, and there is a universal \emph{twisted Thom class}
\[
    u^\tw_n\in H^n\bigl(\twThom(n),\BO(1);\mathcal{L}_\ZZ^{\mathrm{Th}}\bigr).
\]
Of course, an explicit construction of $\twThom(n)$ depends on the model for $\BO(1)$, which we will choose to be $\RR P^\infty$ when the need arises.

To prove the parametrised analogue of Thom's realisability result in Section~\ref{sec:realisability}, we start by introducing twisted cobordism classes $\LL_k(X; \AA \otimes \Or_X)$, whose twisting comes from an integer coefficient system $\AA$ induced by some map $X \to \RR P^\infty$.
We then provide a parametrised version of the Pontryagin--Thom construction to show that there is a bijection
\[
    \bigl[ X^{n+k}, \twThom(n) \bigr]_{\RR P^\infty} \overset{\text{1:1}}{\longleftrightarrow} \LL_k(X; \AA \otimes \Or_X).
\]
This leads to the following parametrised realisability theorem, and thus the main theorem of Section~\ref{sec:realisability}.

\begin{theorem}
     Let $X$ be a closed connected $(n+k)$-manifold with a local system of integer coefficients $\AA$ induced from $\mathcal{L}_\ZZ$ by a map $p\colon X\to \RR P^\infty$. Then a cohomology class ${\alpha\in H^n(X;\AA)}$ is realisable if and only if there is a parametrised map $f\colon X\to \twThom(n)$ over $\RR P^\infty$ such that $f^*(u^\tw_n)=\alpha$.
\end{theorem}

In order to investigate when this condition holds, we observe that $\twThom(n)$ is homotopy equivalent as a retractive space to the homotopy quotient $\MSO(n)\sslash\ZZ_2$  of the universal oriented Thom space $\MSO(n)$, where the latter is regarded as a pointed $\ZZ_2$-space via the action that reverses orientations. This allows us in Section \ref{sec:postnikov_decomposition} to build a Postnikov tower for the twisted Thom space in the category of retractive spaces over $\B\ZZ_2 \simeq \BO(1)$, starting from a $\ZZ_2$-equivariant Postnikov tower of $\MSO(n)$. As $\twThom(n)$ is a non-simple space, the $k$-invariants are twisted cohomology classes, represented by maps into generalised Eilenberg--MacLane spaces. Using the theory of cohomology operations with local coefficients developed by Gitler \cite{Gitler63} and other authors, we are able to compute the first non-trivial \(k\)-invariant to be a twisted analogue of the first non-trivial $k$-invariant for $\MSO(n)$; more specifically, we identify the first non-trivial $k$-invariant with a twisted version of the Steenrod cube operation $\St^5_3 \colon H^*(\placeholder ;\ZZ)\to H^{*+5}(\placeholder ; \ZZ)$. Altogether, this allows us to prove the following result.

\begin{theorem}
    The following statements hold:
    \begin{enumerate}[label=\upshape{(\roman*)}]
        \item For $n=1,2$, every cohomology class $\alpha\in H^n(X;\AA)$ is realisable;
        \item For $n\ge 3$ a necessary condition for $\alpha\in H^n(X;\AA)$ to be realisable is that its twisted Steenrod cube $ \St^5_{3, \tw}(\alpha)\in H^{n+5}(X;\AA)$ is zero. If $\dim(X)\le n+5$, this condition is also sufficient.
    \end{enumerate}
\end{theorem}

Finally in Section \ref{sec:application} we apply our results to give what we believe to be the first known examples of non-realisable integral homology classes in non-orientable manifolds. Our examples are $7$-dimensional homology classes in non-orientable manifolds of dimensions $10$ and $11$, with the former expected to be minimal. They are sufficiently different to include both, although in each case we leverage an example in the untwisted case due to Bohr--Hanke--Kotschick \cite{BohrHankeKotschick00, BohrHankeKotschick02} to show that the twisted operation $ \St^5_{3, \tw}$ applied to the dual cohomology class is nonzero. 

\section{Preliminaries}\label{sec:prelims}

We denote by $\mathcal{K}$ the category of compactly generated weak Hausdorff spaces and by $\mathcal{K}^*$ its pointed version.
In the following, all spaces and maps are in $\mathcal{K}$ unless stated otherwise.

\subsection{Classical computations}
\label{subsec:classical_computations}

We will now briefly summarise the main ideas of Thom's work on realisability for constant coefficients \cite{Thom54}, which inspire and inform the approach taken in this work.
To begin, let us recall that cobordism classes of submanifolds geometrically realise homology classes in a certain range.
Let $X$ be an $(n+k)$-dimensional oriented connected manifold without boundary.

\begin{definition}[see also \cite{Thom54, Atiyah61}]\label{def:oriented_cobordism}
    Two $k$-dimensional closed oriented submanifolds $\iota_i\colon M_i\hookrightarrow X$ ($i=0,1$) are said to be \emph{oriented cobordant} if there is a $(k+1)$-dimensional compact oriented submanifold $W\hookrightarrow X\times [0,1]$ such that 
    \begin{enumerate}[(i)]
        \item $W$ meets $X\times\{i\}$ transversely in $M_i \times \{i\}$,
        \item $W$ has no other boundary, and
        \item the orientation on $\partial W$ agrees with $-M_0 \times \{0\}$ and $M_1 \times \{1\}$.
    \end{enumerate}
    Here $-M_0$ denotes the oriented manifold $M_0$ with the opposite orientation.
    We denote the set of oriented cobordism classes of $k$-dimensional submanifolds of $X$ by $\LL_k(X; \ZZ)$.
\end{definition}

In the stable range, that is, for $n > k+1$, the disjoint union endows \(\LL_k(X;\ZZ)\) with the structure of an abelian group, where the identity is represented by the empty submanifold and inverses are induced by reversing orientations.

An orientation of an embedded $k$-dimensional orientable submanifold $\iota\colon M\hookrightarrow X$ determines a choice of fundamental class $[M]\in H_k(M;\ZZ)$. Consequently, there is a natural map
\begin{equation}\label{eq:rep}
 \mu \colon \LL_k(X;\ZZ) \longrightarrow H_k(X;\ZZ),\qquad [\iota\colon M\hookrightarrow X] \longmapsto \iota_*[M],
\end{equation}
whose image comprises the realisable integer homology classes.
Thus equipped with~\eqref{eq:rep}, the following commutative diagram 
\begin{equation}\label{eq:Thom_rep_diagram}
\begin{tikzcd}[sep=1cm]
    {[X, \MSO(n)]} \arrow[d, "u_\ast"'] \arrow[r, "\cong"', "\text{P--T}"] & \LL_k(X;\ZZ) \arrow[d, "\mu"] \\
    H^n(X; \ZZ) \arrow[r, "\cong"', "\mathrm{PD}"]                 & H_k(X; \ZZ)
\end{tikzcd}
\end{equation}
illustrates the heart of Thom's representability result. By way of elaboration, the left vertical map of \eqref{eq:Thom_rep_diagram} is induced by the map $u \colon \MSO(n) \to K(\ZZ,n)$ representing the universal Thom class, the top horizontal map is given by the Pontryagin--Thom construction, and the bottom horizontal map is given by Poincar\'e duality. By definition, a cohomology class \(\alpha \in H^n(X;\ZZ)\) is realisable if and only if its Poincar\'e dual \(\mathrm{PD}(\alpha) = \alpha \cap [X] \in H_k(X;\ZZ)\) is realisable. But \(\mathrm{PD}(\alpha)\) if realisable precisely when it is in the image of \eqref{eq:rep}. Altogether, one infers from the commutativity of \eqref{eq:Thom_rep_diagram} that $\alpha \in H^n(X;\ZZ)$ is realisable if and only if \(\alpha \in \mathrm{im}(u_\ast)\). In other words, the realisability of a given cohomology class \(\alpha \in H^n(X;\ZZ)\) is equivalent to finding a solution of the lifting problem \eqref{eq:representability_lifting_problem} below: 
\begin{equation}\label{eq:representability_lifting_problem}
\begin{tikzcd}
	& \MSO(n) \arrow[d, "u"]\\
    X \arrow[ur, dashed] \arrow[r, "\alpha"'] & K(\ZZ,n) \mathrlap{.}
\end{tikzcd}\end{equation}

In order to compute the obstructions to lifting in \eqref{eq:representability_lifting_problem}, Thom \cite[Chapitre II]{Thom54} studied the Postnikov tower of the \((n-1)\)-connected space $\MSO(n)$. Indeed, Thom computed that the first non-vanishing homotopy group of \(\MSO(n)\) is $\pi_n(\MSO(n))\cong \ZZ$ and that the map to the bottom Postnikov piece is the Thom class $u \in H^n(\MSO(n);\ZZ)$.
Furthermore, for $n\ge 3$, Thom determined that the first non-trivial $k$-invariant is the so-called \emph{Steenrod cube}
\[
    \St^5_3 \coloneq \beta \, P^1_3 \, \rho_3 \colon K(\ZZ,n) \longrightarrow K(\ZZ,n+5).
\]
Above we have written \(P^1_3 \colon K(\ZZ_3,n) \to K(\ZZ_3,n+4)\) for the reduced Steenrod power operation (as defined in \cite{Steenrod52}), \(\rho_3\) for the mod\;\(3\) reduction, and \(\beta\) for the Bockstein homomorphism associated with the exact coefficient sequence
\[
    0 \longrightarrow\ZZ \overset{\times 3}\longrightarrow \ZZ \overset{\rho_3}\longrightarrow \ZZ_3\longrightarrow 0.
\]

It follows, in particular, that if $X$ is an oriented closed manifold of dimension at most $n+4$, then any cohomology class $\alpha \in H^n(X;\ZZ)$ is realisable. In fact, pushing the calculations a little further, Thom showed that if $n\ge8$ and the dimension of $X$ is at most $n+8$, then $\alpha\in H^n(X;\ZZ)$ is realisable if and only if $\St^5_3(\alpha) = 0$ \cite[Theorem II.17]{Thom54}. 

\subsection{Local coefficient systems}
\label{subsec:loccoeffsys}

We now introduce the notion of a local coefficient system on a topological space and collect various useful facts about such systems.
There are several definitions to choose from, all of which coincide up to isomorphism for nice spaces (e.g.~path-connected manifolds or path-connected CW complexes).
For instance, Steenrod \cite{Steenrod43} and Whitehead \cite[Chapter VI]{Whitehead79} define a local coefficient system on a connected space $X$ to be a contravariant functor \[
    \AA\colon\Pi_1(X)\longrightarrow \catname{Ab}
\]
from the fundamental groupoid of $X$ to the category of abelian groups.
They go on to show that for some basepoint \(x_0 \in X\), the existence of such a functor is equivalent to stipulating a (left) $\pi_1(X,x_0)$-module $A$, see e.g.~\cite[VI: (1.11), (1.12)]{Whitehead79}.
Alternatively, one can view local coefficient systems as locally constant sheaves of abelian groups (see \cite[Chapter 6, Exercises~F]{Spanier95}).

Ultimately, we prefer to think of local coefficient systems as bundles of groups, for this not only allows us to avoid choosing arbitrary base points but also makes a number of naturality and isomorphism statements clear.
Accordingly, we comply with the standard textbook treatments of either Hatcher \cite[\S 3.H, pp.\ 330--336]{Hatcher02} or Davis and Kirk \cite{DavisKirk01} and make the following definitions.

\begin{definition}[Local coefficient system]\label{def:local_coefficient_systems}
    Let $X \in \mathcal{K}$ be a topological space and $A$ be an abelian group viewed as a discrete space.
    A \emph{local coefficient system $\AA$ on $X$} is a fibre bundle $p \colon \AA\to X$ with fibre $A$ and structure group $G\leq \operatorname{Aut} A$.
\end{definition}
    
Note that the fibre $\AA_x\coloneq p^{-1}(x)$ for $x \in X$ is in general non-canonically isomorphic to $A$.
We denote the \emph{trivial local coefficient system} by $\underline{A} \coloneq X \times A$.
\begin{definition}[Morphism of local coefficient systems]
    A \emph{morphism} $\varphi \colon \AA\to \AA'$ of local coefficients systems over $X$ is a commuting diagram
    \[
    \begin{tikzcd}[sep=small]
        \mathcal{A} \arrow[dr] \arrow[rr, "\varphi"] &&  \mathcal{A}' \arrow[dl] \\
                        &           X & 
    \end{tikzcd}
    \]
    such that $\varphi_x \colon \AA_x \to \AA'_x$ is a group homomorphism for each $x \in X$.
\end{definition}

\begin{example}\label{ex:pullback_system}
    If $f\colon X\to Y$ is a continuous map and $\AA \to Y$ is a local coefficient system on $Y$, then the \emph{pullback} $f^*\AA\to X$ is a local coefficient system on $X$. Since the fibres are discrete, by the uniqueness of path lifting a homotopy $H \colon X\times [0,1] \to Y$ from $f_0$ to $f_1$ induces an isomorphism $f_0^*\AA \to f_1^*\AA$ of local coefficient systems.
\end{example}

We define a category $\mathcal{K}_\mathcal{L}$ of local coefficient systems as follows. Objects of $\mathcal{K}_\mathcal{L}$ are triples $(X,C;\AA)$ where $(X,C)$ is a pair of spaces in $\mathcal{K}$ and $\AA$ is a local coefficient system on $X$. A morphism $(f;\varphi)$ in $\mathcal{K}_\mathcal{L}$ from $(X,C;\AA)$ to $(Y,D;\BB)$ consists of a map of pairs ${f \colon (X,C)\to (Y,D)}$ and a morphism of local systems $\varphi \colon \AA\to f^*\BB$.

The singular homology $H_*(X,C;\AA)$ and cohomology $H^*(X,C;\AA)$ groups with local coefficients of a pair of spaces \((X,C)\) are defined in the usual way, see e.g.~\cite{Hatcher02, DavisKirk01}.
In particular, homology with local coefficients is a covariant functor
\[
    H_*(\placeholder,\placeholder;\placeholder) \colon \mathcal{K}_\mathcal{L} \longrightarrow \catname{GrAb}
\]
from $\mathcal{K}_\mathcal{L}$ to graded abelian groups, satisfying the appropriate analogues of the Eilenberg--Steenrod axioms \cite{Whitehead79}. To describe the functoriality of cohomology with local coefficients, we must define a new category $\mathcal{K}_{\mathcal{L}^*}$ with the same objects as $\mathcal{K}_\mathcal{L}$ but with a morphism $(f;\varphi)$ from $(X,C;\AA)$ to $(Y,D;\BB)$ given by a map $f \colon (X,C)\to (Y,D)$ and a morphism of local systems $\varphi \colon  f^*\BB\to\AA$. Accordingly, cohomology with local coefficients is a contravariant functor
\[
    H^*(\placeholder,\placeholder;\placeholder) \colon \mathcal{K}_{\mathcal{L}^*} \longrightarrow \catname{GrAb}
\]
satisfying the ``dualised'' Eilenberg--Steenrod axioms.
In particular, an isomorphism $\varphi \colon \AA\to \BB$ of local coefficient systems over $X$ induces isomorphisms
\begin{align*}
    \varphi_* &= (\id_X,\varphi)_* \colon H_*(X,C;\AA) \stackrel{\cong}{\longrightarrow} H_*(X,C;\BB), \\
    \varphi^* &= (\id_X,\varphi)^* \colon H^*(X,C;\AA) \stackrel{\cong}{\longrightarrow} H^*(X,C;\BB)
\end{align*}
in (co)homology with local coefficients.

To discuss multiplicativity of these homology and cohomology theories, one needs to define tensor products of local coefficient systems. 

\begin{construction}[Tensor product of local coefficient systems]
Let $\AA\to X$ and $\BB\to X$ be local coefficient systems on a space \(X\) with fibres $A$ and $B$, respectively. Take a common trivialising cover $\{U_i\}$ of the local systems so that $\AA$ and \(\BB\) have transition functions 
\[
    g_{ij} \colon U_i\cap U_j\to G \leq\operatorname{Aut}(A) \text{ and } h_{ij} \colon U_i\cap U_j\to H \leq \operatorname{Aut}(B),
\]
respectively. 

We declare the \emph{tensor product of \(\AA\) and \(\BB\)}, denoted $\AA\otimes\BB$, to be the local coefficient system with fibre $A\otimes B\coloneq A\otimes_\ZZ B$ whose transition functions are given by
\[
    g_{ij}\otimes h_{ij} \colon U_i\cap U_j\longrightarrow \Gamma \leq  \operatorname{Aut}(A\otimes B).
\]
Here the structure group \(\Gamma \subseteq \Aut(A \otimes B)\) of the local system \(\AA \otimes \BB\) is given by the image of the natural map 
\[
    G \times H \longrightarrow \Aut(A\otimes B), \qquad (g,h) \longmapsto (a \otimes b \mapsto g(a) \otimes h(b)).
\]
\end{construction}
As a consequence of this construction, one can define cup and cap products with local coefficients
\begin{align*}
    \cup &\colon H^n(X,C;\AA)\otimes H^\ell(X,D;\BB) \longrightarrow H^{n+\ell}(X,C\cup D;\AA\otimes\BB), \\
    \cap &\colon H^n(X,C;\AA)\otimes H_\ell(X,D;\BB) \longrightarrow H_{\ell-n}(X,C\cup D;\AA\otimes\BB),  
\end{align*}
by analogy to the untwisted case. 

We conclude this section with the following useful and well-known facts, cf.~\cite[Exercise 2, p.\ 283]{Spanier95}. 

\begin{prop}\label{prop:local_systems} Denote by \(\underline{\ZZ}\) the trivial local system of integer coefficients. 
\begin{enumerate}[label=\upshape{(\roman*)}] 
    \item The tensor product of local coefficient systems is (up to canonical isomorphism) associative and unital, where the unit is $\underline{\ZZ}$.
    \item If $\AA$ is a local system of integer coefficients, then the tensor square $\AA\otimes\AA$ is canonically isomorphic to $\underline{\ZZ}$. Due to this and the associativity of the tensor product, if $\AA$ and $\BB$ are local systems of integer coefficients, then an isomorphism $\AA\otimes \BB\to \underline{\ZZ}$ furnishes an isomorphism $\BB\to \AA$, and vice-versa.
    \item The tensor product commutes with taking pullbacks in the sense that there is a canonical isomorphism $f^*(\AA\otimes\BB)\cong f^*\AA\otimes f^*\BB$.
\end{enumerate}
\end{prop}

\subsection{Parametrised homotopy theory}\label{subsec:parametrised_homotopy_theory}

In this section, we provide a brief overview of the parametrised homotopy theory utilised throughout this article. For more details, we refer the reader to the book of Crabb and James~\cite{CrabbJames98}; however, in lieu of the classical dialect used there (where ``fibrewise'' is the dominant adjective), we instead opt for the modern terminology used in May--Sigurdsson~\cite{MaySigurdsson06} and Malkiewich~\cite{Malkiewich23}.

\begin{definition}\label{def:parametrised_spaces}
Let \(B\) be a topological space.
We write $\mathcal{K}_{/B}$ for the slice category over $B$.

\begin{enumerate}[label=(\roman*)]
    \item A \emph{parametrised space over $B$} is an object in $\mathcal{K}_{/B}$, that is a space $X$ equipped with a map $p_X \colon X \to B$.
    A \emph{parametrised pair over $B$} is a topological pair $(X,C)$ together with a map $p_X\colon X\to B$; equivalently, it is a topological pair in the slice category $\mathcal{K}_{/B}$.
    A \emph{parametrised CW-pair over $B$} is a parametrised pair $(X,C; p_X)$ over $B$ for which $(X,C)$ is a CW-pair in the usual sense.
    
    \item A \emph{parametrised map of pairs over $B$} is a morphism of pairs in $\mathcal{K}_{/B}$, that is a map of pairs
    \[
         f \colon (X,C)\longrightarrow (Y,D),
    \]
    satisfying $p_Y \circ f = p_X$.
    Homotopies, homotopy classes, and homotopy equivalences over $B$ are understood in the slice category $\mathcal{K}_{/B}$.
    We write
    \[
        [(X,C),(Y,D)]_B
    \]
    for the set of homotopy classes of maps of pairs over $B$.

    \item A parametrised space $(X;p_X)$ is called \emph{fibrant} if $p_X \colon X \to B$ is a (Hurewicz) fibration in $\mathcal{K}$.
    
    \item A \emph{retractive space over $B$} is a parametrised space $(X; p_X)$ over $B$ together with a section $s \colon B \to X$, namely
    \[
    \begin{tikzcd}
        B \arrow[r, "s_X"] \arrow[rr, bend right, "\id_B"'] & X \arrow[r, "p_X"] & B \mathrlap{.}
    \end{tikzcd}
    \]
    A \emph{map of retractive spaces} is a parametrised map over $B$ which also preserves the chosen sections.
    We denote the category of retractive spaces over $B$ by~$\mathcal{K}^*_{/B}$.

    \item If $Y$ is a retractive space over $B$ and $(X,C; p_X)$ is a parametrised pair over $B$, we set
    \[
        [(X,C),Y]^\mathrm{rel}_B \coloneq [(X,C),(Y,s_Y(B))]_B
    \]
    Thus $[(X,C),Y]^\mathrm{rel}_B$ consists of homotopy classes over $B$ of maps $f\colon X\to Y$ such that $f(C)\subseteq s_Y(B)$.
    We write $\simeq_B$ for the corresponding homotopy relation.
\end{enumerate}
\end{definition}

The following statements of Lemma \ref{lem:forgetful_homotopy_classes_bijection} and Proposition \ref{prop:n-equiv} are instances of \cite[Theorem~3.1 \& 3.2]{JamesThomas66} adapted to our particular situation. 

\begin{lemma}\label{lem:forgetful_homotopy_classes_bijection}
    Let $(Y; p_Y, s_Y)$ be a fibrant retractive space over $B$.
    Let $(X,C; p_X)$ be a parametrised pair over $B$, and assume that $A \hookrightarrow X$ is a cofibration.
    Then the forgetful map
    \[
        [(X,C), Y]_B^\mathrm{rel} \longrightarrow (p_Y)_*^{-1}([p_X]), \qquad [f]_B \longmapsto [f]
    \]
    is a bijection, where $(p_Y)_* \colon [(X,C), (Y, s_Y(B))] \to [(X,C), (B, B)]$ is the induced map on usual homotopy classes of maps of pairs so that \[(p_Y)_*^{-1}([p_X]) = \{[f] \in [(X,C), (Y,s_Y(B))] \mid p_Y \circ f \simeq p_X\}.\]
\end{lemma}
\begin{proof}
    The proof is a standard exercise in applying the homotopy lifting extension property.
\end{proof}

\begin{prop}\label{prop:n-equiv}
    Let $(Y;p_Y, s_Y)$ and $(Z;p_Z, s_Z)$ be fibrant retractive spaces over $B$ and let ${f \colon Y \to Z}$ be a map of retractive spaces.
    Suppose that $Y$ and $Z$ are connected and assume that the induced map
    \[
        \pi_i(f) \colon \pi_i(Y) \longrightarrow \pi_i(Z)
    \]
    is an isomorphism for $i < n$ and an epimorphism for $i = n$.
    If $(X,C; p_X)$ is a parametrised CW-pair over $B$ of dimension $\dim(X-C) = d$, then the induced map
    \[
        f_* \colon [(X,C), Y]_B^\mathrm{rel} \longrightarrow [(X,C), Z]_B^\mathrm{rel}
    \]
    is a bijection if $d < n$ and an epimorphism if $d = n$.
\end{prop}
\begin{proof}
    First, note that the following diagram of sets of usual homotopy classes of maps of pairs commutes:
    \[
    \begin{tikzcd}[column sep={6em,between origins}]
        {[(X,C), (Y,s_Y(B))]} \arrow[rr, "f_*"] \arrow[dr, "(p_Y)_*"'] & & {[(X,C), (Z,s_Z(B))]} \arrow[dl, "(p_Z)_*"] \\
        & {[(X,C),(B,B)]} \mathrlap{.} &
    \end{tikzcd}
    \]
    Thus $f_*$ restricts to a map $\Psi_f \colon (p_Y)_*^{-1}([p_X]) \to (p_Z)_*^{-1}([p_X])$.
    Moreover, by classical results, the map $\Psi_f$ is bijective for $d < n$ and surjective for $d=n$.
    Second, apply Lemma~\ref{lem:forgetful_homotopy_classes_bijection} to obtain the following chain of maps:
    \begin{equation}\label{eq:chainofmaps}
    \begin{tikzcd}[sep=small]
        {[(X,C), Y]_B^\mathrm{rel}} \arrow[r, "\cong"] & (p_Y)_*^{-1}([p_X]) \arrow[r, "\Psi_f"] & (p_Z)_*^{-1}([p_X]) & {[(X,C), Z]_B^\mathrm{rel}} \mathrlap{.} \arrow[l, "\cong"']
    \end{tikzcd}
    \end{equation}
    The result now follows from a routine check that the composition \eqref{eq:chainofmaps} is given precisely by $f_* \colon [(X,C), Y]_B^\mathrm{rel} \to [(X,C), Z]_B^\mathrm{rel}$.
\end{proof}

\begin{lemma}\label{lem:fibrewisepushouts}
    Let $g \colon (X;p_X)\to (Y;p_Y)$ be a parametrised map in $\mathcal{K}_{/B}$. Form the pushout
    \[
    \begin{tikzcd}
        X \arrow[r,"p_X"] \arrow[d,"g",swap] & B \arrow[d,"\sigma"] \\
        Y \arrow[r] & Z \arrow[ul, phantom, "\ulcorner", very near start]
    \end{tikzcd}
    \]
    in the category $\mathcal{K}_{/B}$.
    Then the obvious map $p_Z \colon Z \to B$, induced by the pushout, turns $(Z; p_Z, \sigma)$ into a retractive space over $B$.
    Moreover, if $g$ is a closed cofibration over $B$ and $(X;p_X)$ and $(Y;p_Y)$ are fibrant, then $(Z;p_Z)$ is fibrant.
\end{lemma}
\begin{proof}
    That $(Z; p_Z, \sigma)$ is a retractive space follows immediately from the definitions.
    Note that in order to take the pushout of spaces over $B$, we need only take the ordinary pushout since it comes ready-equipped with a map to $B$.
    Finally, the fact that $p_Z$ is a fibration if $g$ is a cofibration and $p_X$ and $p_Y$ are fibrations follows from \cite[Prop.~1.3]{Clapp81}.
\end{proof}

\subsection{Homotopy quotients as parametrised spaces}
\label{subsec:quillen_pair_parametrised_Gspaces}

One important class of parametrised spaces are homotopy quotients of $G$-spaces, where $G$ is a topological group. As it turns out, the homotopy quotient construction provides the left adjoint of a Quillen pair (for appropriately chosen model structures).

If $G$ is a topological group, let $\mathcal{K}_G$ denote the category of $G$-spaces and $G$-maps, and let $\mathcal{K}^*_G$ denote its pointed version, where basepoints are assumed to be $G$-fixed. For each \(\mathcal{C} \in \{\mathcal{K}_G,\mathcal{K}^*_G,\mathcal{K}_{/B},\mathcal{K}_{/B}^*\}\), observe that there is a forgetful functor $U \colon \mathcal{C} \to \mathcal{K}$.
Accordingly, we say that a map $f$ in \(\mathcal{C}\) is a \emph{fibration}, respectively \emph{$n$-equivalence}, respectively \emph{weak equivalence} if the map $U(f)$ is a fibration, resp.\ $n$-equivalence, resp.\ weak equivalence in $\mathcal{K}$.%
\footnote{
    Recall that an $n$-equivalence is a map $f \colon X \to Y$ for which the induced map $\pi_i(f)$ on homotopy groups is an isomorphism for $i < n$ and an epimorphism for $i=n$.
}
For example, with this convention, the fibrations in $\mathcal{K}_G$ are the $G$-equivariant fibrations rather than the $G$-fibrations. In other words, we equip \(\mathcal{C}\) with the naive (or Borel) model structure, cf.\ \cite[Theorem 1.6 with \(\mathscr{F}=\{e\}\)]{GMR2018}.

Fix a model for $\E G$ such that the projection $\E G\to \B G=\E G/G$ is a numerable principal $G$-bundle. Then the \emph{homotopy quotient} of a $G$-space $X$ is defined by 
\[
    X\sslash G \coloneq \E G \times_G X,
\]
where on the right we take the quotient by the diagonal action on $\E G \times X$.
It comes equipped with a fibration $p_X \colon X \sslash G\to \B G$ induced by collapsing $X$ to a point.
If $X$ has a $G$-fixed point $x_0$, then $p_X$ admits a section $s_X \colon \B G\to X\sslash G$ given by $s_X[e]=[e, x_0]$.
These constructions are functorial, thus we may regard $(\placeholder)\sslash G$ as a functor $\mathcal{K}_G \to \mathcal{K}$, or $\mathcal{K}_G\to\mathcal{K}_{/\B G}$, or $\mathcal{K}_G^*\to\mathcal{K}^*_{/\B G}$.

\begin{lemma}\label{lem:propsHQ}
    The homotopy quotient functor has the following properties:
    \begin{enumerate}[label=\upshape{(\roman*)}]
       \item $(\placeholder)\sslash G \colon \mathcal{K}_G\to\mathcal{K}$ admits a right adjoint, in particular preserves pushouts;
       \item $(\placeholder)\sslash G \colon \mathcal{K}_G\to\mathcal{K}_{/\B G}$ admits a left adjoint, in particular preserves pullbacks;
       \item $(\placeholder)\sslash G \colon \mathcal{K}^*_G\to\mathcal{K}^*_{/\B G}$ admits a left adjoint, in particular preserves pullbacks.
       \item $(\placeholder)\sslash G$ preserves fibrations, $n$-equivalences and weak equivalences.
    \end{enumerate}
\end{lemma}
\needspace{2\baselineskip}
\begin{proof}
    \leavevmode
    \begin{enumerate}[(i)]
        \item $(\placeholder)\sslash G \colon \mathcal{K}_G\to\mathcal{K}$ is the composition of the functor $\E G \times (\placeholder) \colon \mathcal{K}_G\to \mathcal{K}_G$ and the strict orbits functor $(\placeholder)/G \colon \mathcal{K}_G\to\mathcal{K}$.
        The former admits a right adjoint
        \[
            \operatorname{Map}(\E G,\placeholder) \colon \mathcal{K}_G \longrightarrow \mathcal{K}_G,
        \]
        where we take the space of not necessarily equivariant maps endowed with the compact-open topology and the conjugation $G$-action.
        The strict orbits functor admits a right adjoint given by
        \[
            (\placeholder)_{\operatorname{triv}} \colon \mathcal{K}\longrightarrow \mathcal{K}_G,
        \]
        which endows a space with the trivial $G$-action.
        \item Following Dror--Dwyer--Kan \cite[Proposition 2.3]{DrorDwyerKan80} (who work in the simplicial setting), we define a functor $A \colon \mathcal{K}_{/\B G}\to\mathcal{K}_G$ as follows.
        Given a parametrised space $(X,p_X)$ over $\B G$, let $A(X)$ be the total space of the principal $G$-bundle over $X$ obtained by pulling back $\E G\to \B G$ via the map $p_X \colon X\to \B G$. Functoriality of $A$ follows from the functoriality of pullbacks.
        The adjunction map
        \[
            \mathcal{K}_G(A(X),Y) \longrightarrow \mathcal{K}_{/\B G}(X,Y\sslash G)
        \]
        is described as follows: given a $G$-map $\varphi \colon A(X) \to Y$, we obtain a $G$-map
        \[
            (\varphi,\bar{p}_X) \colon A(X)\longrightarrow \E G \times Y,
        \]
        where $\bar{p}_X \colon p_X^* \E G \to \E G$ is the corresponding map covering $p_X$.
        After taking quotients with $G$, this gives a map $X = A(X)/G\to Y\sslash G$ of spaces over $\B G$.
        We omit the rest of the details.
        \item We define $A^* \colon \mathcal{K}^*_{/\B G}\to\mathcal{K}^*_G$ by a slight modification of the above.
        Given a retractive space $(X,p_X,s_X)$ over $\B G$, we have a double pullback diagram
        \[
        \begin{tikzcd}
            \E G \arrow[r,"\bar{s}_X"] \arrow[d] \arrow[rd, phantom, "\lrcorner", very near start] & A(X) \arrow[r,"\bar{p}_X"] \arrow[d] \arrow[rd, phantom, "\lrcorner", very near start] & \E G \arrow[d] \\
            \B G \arrow[r,"s_X"] & X \arrow[r,"p_X"] & \B G \mathrlap{.}
        \end{tikzcd}
        \]
        Collapsing $\bar{s}_X(\E G)$ to a point, or equivalently taking the pushout of $\ast \leftarrow \E G \stackrel{\bar{s}_X}{\to} A(X)$ in the category of $G$-spaces, we obtain a pointed $G$-space $A(X)_0\in \mathcal{K}^*_G$.
        Now a pointed $G$-map $\varphi_0 \colon A(X)_0\to Y$ corresponds to a $G$-map $\varphi \colon A(X)\to Y$ which maps $\bar{s}_X(EG)$ to the base point, and gives a section-preserving map $X = A(X)/G \to Y\sslash G$.
        The rest of the details are left to the reader.
        \item See \cite{EbertMO} for the fact that $(\placeholder)\sslash G$ is fibration-preserving.
        The rest follows easily from the $5$-lemma applied to the long exact homotopy sequences of the fibrations
        \[
            (\placeholder)\longrightarrow(\placeholder)\sslash G\longrightarrow \B G.
        \]
    \end{enumerate}
    This completes the proof. 
\end{proof}

\subsection{Generalised Eilenberg--MacLane spaces}
\label{subsec:genEMspaces}

Classically, ordinary cohomology with constant coefficients is represented by Eilenberg--MacLane spaces; however, in the presence of a local coefficient system, one must replace the Eilenberg--MacLane space with a retractive space that encodes the coefficient twisting.
We recall the construction here and refer the reader to e.g.~\cite[\S 7]{Gitler63} or \cite{Robinson72} for more details.

Let $A$ be an abelian group and let $G \le \Aut(A)$ be a subgroup.
For $n \ge 1$, we define the \emph{generalised Eilenberg--MacLane space of type $(A,G,n)$} as
\[
    L_G(A,n) \coloneq K(A,n) \sslash G.
\]
Here, $K(A,n)$ is some simplicial model of the Eilenberg--MacLane space that admits an action of $G$ induced by the action on $\pi_n(K(A,n))\cong A$, cf.\ \cite[p.\ 173]{Gitler63}, and $(\placeholder) \sslash G$ is the homotopy quotient functor from Subsection~\ref{subsec:quillen_pair_parametrised_Gspaces}.
Since the canonical basepoint \( * \in K(A,n) \) is preserved by the $G$-action, it follows from the previous discussion that the homotopy fibration \( p \colon L_G(A,n) \to \B G \) admits a canonical section
\[
    \sigma \colon \B G \to L_G(A,n).
\]
As such, the generalised Eilenberg--MacLane space is the retractive space $(L_G(A,n),p, \sigma)$ over~$\B G$.

\begin{definition}[Universal local coefficient system]\label{def:universal_coefficient_system}
    For an abelian group $A$ and subgroup \( G \le \Aut(A) \), there is a \emph{universal local coefficient system} on $\B G$ defined as the balanced product
    \begin{align*}
        \mathcal{L}_{(A,G)} \coloneq \E G \times_{G} A.
    \end{align*}
    Via pullback, see Example~\ref{ex:pullback_system}, the generalised Eilenberg--MacLane space $L_G(A,n)$ admits a universal coefficient system given by $\mathcal{L}_{(A,G)}^n \coloneq p^*\mathcal{L}_{(A,G)}$.
\end{definition}

\begin{remark}\label{rem:F_AG}
    For a fixed abelian group $A$ and subgroup $G \le \Aut(A)$, there is a functor
    \begin{align*}
        F_{(A,G)} &\colon \mathcal{K}_{/\B G} \longrightarrow \mathcal{K}_{\mathcal{L}}, \qquad (X; p_X) \longmapsto (X,\varnothing; p_X^*\mathcal{L}_{(A,G)}).
    \end{align*}
    from spaces over $\B G$ to pairs of spaces with local coefficient systems.
    Because principal $G$-bundles are classified by maps into $\B G$, the essential image of $F_{(A,G)}$ consists of all local coefficient systems $\AA \to X$ with fibre $A$ and structure group $G \le \Aut(A)$.
\end{remark}

We have the following instance of parametrised Brown representability.

\begin{prop}[\cite{Gitler63, Robinson72}]\label{prop:twisted_representability}
     For each $n\geq 1$ there exists a fundamental class
     \[
        e_n \in H^n\bigl(L_G(A,n),\sigma(\B G);\mathcal{L}_{(A,G)}^n \bigr)
    \]
    such that for every parametrised CW-pair $(X,C; p_X)$ over $\B G$ the map
    \begin{align*}
        [(X,C),L_G(A,n)]^\mathrm{rel}_{\B G} &\longrightarrow H^n(X,C;\AA), \\
        f &\longmapsto f^*e_n
    \end{align*}
    is a natural bijection. Here $\AA = p_X^* \mathcal{L}_{(A,G)}$ is the induced local coefficient system over $X$.
\end{prop}

\begin{notate}\label{notate:L_A}
    If the subgroup $G \le \Aut(A)$ is clear, we may omit $G$ from our notation entirely, that is we may write $\mathcal{L}_A$, $\mathcal{L}_A^n$ rather than $\mathcal{L}^n_{(A,G)}$ and $L(A,n)$ instead of \(L_G(A,n)\). 
\end{notate}

\section{Twisted Thom theory and the twisted Thom space}\label{sec:twisted_Thom_theory}

In this section, we recall orientation systems, the twisted Thom isomorphism theorem, and twisted Umkehr homomorphisms.
Thereafter, we introduce the new twisted Thom space and its twisted Thom class. 

\begin{notate}[Normal bundle of an embedding]\label{notate:normal_bundle}
In the following, if $\iota \colon M \to X$ is an embedding of smooth manifolds, we denote its normal bundle by $\nu\iota$ or $\nu M$.
\end{notate}

\subsection{Orientation systems}

Throughout the article, we will frequently work with the following local coefficient systems defined for manifolds and vector bundles.
The next definition is standard.

\begin{definition}[Orientation local system of a manifold]\label{def:orientation_system_manifold} 
    Let $X$ be a connected $n$-manifold. The \emph{orientation local system} $\Or_X$ of $X$ is described as follows.
    Given an open set $V\subseteq X$ and $x\in V$, let $j_x:(X,X-V)\hookrightarrow (X,X-x)$ be the inclusion.
    The total space of $\Or_X$ is the union $\bigsqcup_{x\in X} H_n(X,X-x;\ZZ)$ of local homology groups, topologised by taking as a basis the sets $$U_\alpha=\{(j_x)_*(\alpha) \in H_n(X,X-x;\ZZ)\mid x\in U\}$$ for each open set $U\subseteq X$ and $\alpha\in H_n(X,X-U;\ZZ)$.
    The obvious projection $\Or_X\to X$ is then a fibre bundle with fibre $\ZZ$, a global section of which is an orientation of $X$.

    One can alternatively view $\Or_X$ as the locally constant sheaf associated with the presheaf given by $U\mapsto H_n(X,X-U;\ZZ)$, see \cite[VI.1, Example 5]{Whitehead79}.
\end{definition}

The next definition introduces orientation local systems of vector bundles and is less frequently used in the literature. Before that, however, we establish the following notation.

\begin{notate}[Zero section of a vector bundle]\label{notate:zero_section}
    If $E\to X$ is a vector bundle over a topological space $X \in \mathcal{K}$, then we denote the zero section by $z_E \colon X \to E$; although, we will frequently omit the subscript and simply write $z \colon X \to E$ when the context is clear.  
\end{notate}

\begin{definition}[Orientation local system of a vector bundle] \label{def:orientation_system_bundle}
    Let $E\to X$ be a rank $n>0$ vector bundle over a locally contractible space~$X$.
    For each $x\in X$ the relative cohomology of the fibre $H^n(E_x,E_x-z(x);\ZZ)$ is infinite cyclic.
    Put a topology on the union $\Or_E\coloneq\bigsqcup_{x\in X} H^n(E_x,E_x-z(x);\ZZ)$ by taking as basic opens the sets
    \[
        U_\alpha = \{ (j_x)^*(\alpha) \in H^n(E_x,E_x-z(x);\ZZ)\mid x\in U\}
    \]
    for each open set $U\subseteq X$ and cohomology class $\alpha \in H^n(E_U,E_U-z(U);\ZZ)$.
    Here
    \[
        j_x \colon (E_x,E_x-z(x))\hookrightarrow (E_U,E_U-z(U))
    \]
    is the inclusion of the fibre over $x$.
    We then call the obvious projection $\Or_E\to X$ the \emph{orientation local system of $E$}.
\end{definition}

\begin{remark}\label{rmk:local_thom_class}
    The following useful observations will be freely used in the sequel. 
    Suppose that
    \[
        \varphi \colon E_U \stackrel{\cong}\longrightarrow U \times \RR^n
    \]
    is a local trivialisation over some open neighbourhood $U \subset X$.
    Let
    \[
        u_{\RR^n}\in H^n(\RR^n,\RR^n-\{0\})
    \]
    denote the generator determined by the standard orientation of $\RR^n$.
    Then: 
    \begin{enumerate}[(a)]
        \item the trivialisation $\varphi$ determines a class
        \[
            u^\varphi_U \coloneq \varphi^* \pr_2^* (u_{\RR^n}) \in H^n(E_U, E_U - z(U)),
        \]
        which, by definition, restricts to generators
        \[
            u_x^\varphi \coloneq j_x^* u_U^\varphi \in H^n(E_x, E_x - z(x)) \cong H^n(\RR^n, \RR^n - \{0\})
        \]
        for each $x \in U$;
        \item if $U$ is contractible, the restriction homomorphism
        \[
            j_x^* \colon H^n(E_U, E_U - z(U)) \to H^n(E_x, E_x - z(x))
        \]
        is an isomorphism for each $x \in U$ by homotopy invariance;
        \item if $U$ is contractible, the class $u^\varphi_U \in H^n(E_U, E_U - z(U))$ in Observation\,(a) is a generator according to Observation\,(b); and 
        \item if $\varphi' \colon E_U \to U \times \RR^n$ is another trivialisation with transition map
        \[
            A \colon U \to \operatorname{GL}(\RR^n),
        \]
        then we have that
        \[
            u_x^{\varphi'} = \operatorname{sign}\det(A(x))\, u_x^\varphi
        \]
        for each $x \in U$.
    \end{enumerate}
\end{remark}

Next we collect several handy canonical identifications involving the orientation local systems described in Definition~\ref{def:orientation_system_manifold} and Definition~\ref{def:orientation_system_bundle} above.
Although the existence of the following isomorphisms is easily believed, we will use the fact that there are no choices involved during the slightly technical derivations carried out in Section \ref{subsec:twisted_Umkehr}.

\begin{lemma} \label{lemma:prop_loc_sys}
    Let $E\to X$ be a vector bundle of rank $n >0$ over a locally contractible space~$X$.
    \begin{enumerate}[label=\upshape{(\roman*)}]
        \item If $\det E = \bwedge^n\! E$ is the determinant line bundle of $E$, then there is an isomorphism $$\Or_E \cong \Or_{\det E};$$ 
        \item If $f \colon Y\to X$ is a continuous map, then there is an isomorphism $$f^*\Or_E\cong \Or_{f^*E};$$
        \item If $F\to X$ is a second vector bundle of rank $\ell$, then there is an isomorphism $$\Or_E\otimes \Or_F\cong \Or_{E\oplus F};$$ 
        \item If $X$ is a connected $n$-manifold with tangent bundle $TX\to X$, then there is an isomorphism $$\Or_X\cong \Or_{TX}.$$
    \end{enumerate}
    Moreover, all isomorphisms \emph{(i)--(iv)} above are canonical.
\end{lemma}

\begin{proof}
    Throughout the proof, we omit untwisted integer coefficients from the notation for the sake of convenience.
    Since $X$ is locally contractible, for each $x_0 \in X$, there exists a contractible open neighbourhood $U \subset X$ of $x_0$.
    Over such a contractible neighbourhood $U$, we construct each isomorphism locally using the generators $u_U^\varphi$ for a given trivialisation $\varphi \colon E_U \to U \times \RR^n$ as discussed in Remark~\ref{rmk:local_thom_class}.
    \begin{enumerate}[(i)]
        \item If $\varphi \colon E_U \to U \times \RR^n$ is a trivialisation over a contractible neighbourhood $U \subset X$, then there is an induced trivialisation
        \[
            \widetilde{\varphi} \colon \det E_U \longrightarrow U \times \RR
        \]
        given by the $n$-th exterior power.
        Consequently, there is an isomorphism
        \begin{align*}
            H^n(E_U, E_U - z(U)) &\stackrel{\cong}\longrightarrow H^1(\det E_U, \det E_U - z(U)),\\
                u_U^\varphi &\longmapsto u_U^{\widetilde{\varphi}},
        \end{align*}
        which does not depend on the choice of trivialisation.
        This yields the desired canonical isomorphism $\Or_E \cong \Or_{\det E}$.
        \item This follows from Definition~\ref{def:orientation_system_bundle} and the definition of the pullback.
        \item Over a contractible neighbourhood $U$, let $\varphi \colon E_U \to U \times \RR^n$ and $\psi \colon F_U \to U \times \RR^\ell$ be local trivialisations of $E$ and $F$, respectively.
        Then there is an induced trivialisation
        \[
            \varphi \oplus \psi \colon (E \oplus F)_U \to U \times (\RR^{n} \oplus \RR^{\ell}) \cong U \times \RR^{n+\ell}
        \]
        of the Whitney sum.
        The desired isomorphism is then given by
        \begin{align*}
            H^n(E_U, E_U - z_E(U)) \otimes H^\ell(F_U, F_U-z_F(U)) &\stackrel{\cong}\longrightarrow H^{n+\ell}((E \oplus F)_U, (E \oplus F)_U - z_{E\oplus F}(U)), \\
            u_U^\varphi \otimes u_U^\psi &\longmapsto u_U^{\varphi \oplus \psi}.
        \end{align*}
        A direct check shows that the isomorphism does not depend on the choice of trivialisations,
        which establishes that $\Or_E\otimes \Or_F\cong \Or_{E\oplus F}$, as required.
        \item Let $V \subset X$ be a coordinate neighbourhood of $X$ with coordinate chart
        \(
            \varphi \colon V \longrightarrow \RR^n.
        \)
        Further let $B \subset \varphi(V) \subset \RR^n$ be an open ball with $\overline{B} \subset \varphi(V)$ so that the pre-image $U \coloneq \varphi^{-1}(B)$ is a contractible neighbourhood in $X$.
        Under these assumptions, we see (after two applications of the excision isomorphism) that the chart $\varphi$ determines a generator
        \[
            \theta_U^\varphi \in H_n(X, X-U) \cong H_n(V, V-U) \cong H_n(\RR^n, \RR^n -B) \cong \ZZ
        \]
        that restricts to generators
        \[
            \theta_x^\varphi \coloneq (j_x)_* \theta_U^\varphi \in H_n(X, X-x),
        \]
        for any $x \in U$, where $(j_x)_*$ is an isomorphism.
        Moreover, the differential of $\varphi$ induces a trivialisation
        \[
            \mathrm d\varphi \colon TX_U \stackrel{\cong}\longrightarrow U\times \RR^n
        \]
        and thereby a generator
        \[
            u_U^{\mathrm d\varphi} \in H^n(TX_U,TX_U-z(U)).
        \]
        Combining the above, we obtain the perfect pairing
        \begin{align*}
            H^n(TX_U, TX_U - z(U)) \otimes H_n(X, X-U) &\longrightarrow \ZZ, \\
            u^{\mathrm{d}\varphi}_U \otimes \theta_{U}^\varphi &\longmapsto 1,
        \end{align*}
        which, since a change of coordinates multiplies both $\theta_{U}^\varphi$ and $u_U^{\mathrm d\varphi}$ by the same sign, is independent of the choice of chart $\varphi$.
        Furthermore, considering that 
        \[
            (j_x)_* \theta_U^\varphi = \theta_x^\varphi \quad \text{and} \quad j_x^* u_U^{\mathrm d\varphi} = u_x^{\mathrm d\varphi},
        \] these local pairings glue to a continuous global isomorphism and subsequently produce a well-defined isomorphism
        \begin{equation}\label{eq:TXxX=Z}
            \Or_{TX}\otimes \Or_X \cong \underline{\ZZ}.
        \end{equation}
        Finally, tensoring \eqref{eq:TXxX=Z} with $\Or_{TX}$ establishes the claimed canonical isomorphism identifying \(\Or_{X}\) and \(\Or_{TX}.\)
    \end{enumerate}
    This finishes the proof.
\end{proof}

\begin{remark}\label{rmk:universal_orientation_system}
    Using orientation systems, we can view the universal local coefficient system $\mathcal{L}_\ZZ$ for $A=\ZZ$ and structure group $G=\ZZ_2$, as defined in Definition~\ref{def:universal_coefficient_system}, in purely geometric terms.
    The classifying space $\B\ZZ_2$ is homotopy equivalent to $\BO(1)$, which admits a universal line bundle $\gamma_1$.
    As such, we can identify $\mathcal{L}_\ZZ$ with the \emph{universal orientation system} $\Or_{\gamma_1}$.
\end{remark}

\subsection{The twisted Thom isomorphism}
\label{subsec:twistedThomisos}

Let $\pi \colon E\to X$ be a vector bundle of rank $n$ over a paracompact space $X$, with zero section $z \colon X\to E$.
Recall that the orientation local system $\Or_E\to X$ has fibres $(\Or_E)_x\cong H^n(E_x,E_x-z(x);\ZZ)\cong \mathbb{Z}$.
\begin{definition}[Twisted Thom class] \label{defn:twisted_thom_class}
    A \emph{twisted Thom class} for $E$ is a cohomology class
    \[
        u_E\in H^n(E,E-z(X);\pi^*\Or_E)
    \]
    such that for each fibre inclusion $j_x \colon (E_x,E_x-z(x))\hookrightarrow (E,E-z(X))$ the restriction $(j_x)^*(u_E)$ is a generator of
    \begin{align*}
    H^n(E_x,E_x-z(x);j_x^*\pi^*\Or_E) & = H^n(E_x,E_x-z(x);(\Or_E)_x)\\
    & \cong H^n(E_x,E_x-z(x);H^n(E_x,E_x-z(x);\ZZ)).
    \end{align*}
\end{definition}

According to Thom \cite{Thom52} such a Thom class always exists. Indeed, note that, by the Universal Coefficient Theorem, the latter group is isomorphic to 
\begin{multline*}
    \Hom(H_n(E_x,E_x-z(x);\ZZ),\Hom(H_n(E_x,E_x-z(x);\ZZ),\ZZ)) \\
     \cong \Hom (H_n(E_x,E_x-z(x);\ZZ)\otimes H_n(E_x,E_x-z(x);\ZZ), \ZZ),
\end{multline*}
hence has a preferred generator given by the canonical isomorphism $a\otimes a\mapsto 1$ for any choice of generator $a\in H_n(E_x,E_x-z(x);\ZZ)\cong \ZZ$.
Altogether, there is a canonical choice of Thom class $u_E$, which we refer to as \underline{the} twisted Thom class.

Given a local coefficient system $\BB$ on $X$, Thom \cite{Thom52} has shown that the homomorphism
\begin{equation}\label{eq:twisted_thom_iso_I}
    \varphi^* \colon H^*(X;\BB) \to H^{*+n}(E,E-z(X);\pi^*\BB\otimes \pi^*\Or_E)
\end{equation}
given by $\varphi^*(\alpha)=\pi^*(\alpha)\cup u_E$ is an isomorphism, the so-called \emph{twisted Thom isomorphism}. Using the canonical isomorphism $\pi^*\Or_E\otimes \pi^*\Or_E\cong\underline{\ZZ}$ one also deduces an isomorphism
\begin{equation}\label{eq:twisted_thom_iso_II}
    \varphi^* \colon H^*(X;\BB\otimes\Or_E) \stackrel{\cong}{\longrightarrow} H^{*+n}(E,E-z(X);\pi^*\BB).
\end{equation}
The twisted Thom class and the twisted Thom isomorphism have the expected naturality properties with respect to maps of vector bundles.

\subsection{Twisted Umkehr homomorphisms}
\label{subsec:twisted_Umkehr}

We now define an Umkehr map in cohomology with local coefficients and prove a push-pull formula that can be used to derive our main result.
To this end, let $X$ be a compact manifold, and let $\iota \colon M\hookrightarrow X$ be an embedding of codimension $n$ such that the image $\iota(M)$ is contained in the interior of $X$. Further let $\AA$ be an arbitrary local coefficient system over $X$.
The following is a simplification of the definition in Section~2 of~\cite{OhmotoSaekiSakuma03}, which treats arbitrary smooth maps.

\begin{definition}\label{defn:Umkehr_map} Let \(\iota \colon M \xhookrightarrow{} X\) and \(\mathcal{A}\) be as above. Then the \emph{Umkehr map}
    \[
        \iota^! \colon H^*(M;\iota^*\AA\otimes\mathcal{O}_{\nu\iota})\to H^{*+n}(X,\partial X;\AA)
    \]
    is defined as follows.
    Choose a tubular neighbourhood $U$ of $M$ in $X$ which is identified with the total space of $\nu\iota \cong U$.
    Let $p_U \colon U\to M$ be the projection, and $\jmath_U \colon U\hookrightarrow X$ the inclusion.
    Note that the radial homotopy $\iota\circ p_U \simeq \jmath_U$  induces a canonical isomorphism $p_U^*\iota^*\AA\cong \jmath_U^*\AA$.
    Then $\iota^!$ is the composition
    \begin{align*}
        H^*(M;\iota^*\AA\otimes\mathcal{O}_{\nu\iota}) & \stackrel{\varphi^*}{\;\longrightarrow\;} H^{*+n}(U,U-M;p_U^*\iota^*\AA)\\
        & \stackrel{\cong}{\;\longrightarrow\;}  H^{*+n}(U,U-M;\jmath_U^*\AA) \\
        & \stackrel{(\jmath_U^*)^{-1}}{\;\longrightarrow\;} H^{*+n}(X,X-M;\AA) \\
        & \;\longrightarrow\; H^{*+n}(X,\partial X;\AA).
    \end{align*}
    Here, the first map is the twisted Thom isomorphism, the map $\jmath_U^*$ is an isomorphism by excision, and the unlabelled map is induced by an inclusion.%
    \footnote{
        Note that the definition of the Umkehr map does not depend on the choice of tubular neighbourhood since such neighbourhoods are unique up to ambient isotopy, see e.g.~\cite[Theorem 7.4.4, p.\ 213]{Mukherjee15}.
    }
\end{definition}

\begin{lemma}\label{lem:ThomUmkehr}
    Let $\pi \colon E\to M$ be a vector bundle of rank $n$ over a closed manifold $M$, and let $z\colon M\to D(E)$ denote the zero section of its disc bundle with respect to some Riemannian metric. Let $1 \in H^0(M;\mathbb{Z})\cong H^0(M;z^*\pi^*\Or_E\otimes \Or_E)$ denote the unit element. Then
    \[
        z^!(1) \in H^n(D(E),S(E);\pi^*\Or_{E}),
    \]
    corresponds to the twisted Thom class $u_E\in H^n(E,E-z(M);\pi^*\Or_E)$ under the isomorphism induced by the inclusion $(D(E),S(E))\hookrightarrow (E,E-z(M))$.
\end{lemma}
\begin{proof}
    The proof follows immediately from the definition of the twisted Umkehr map.
\end{proof}

\begin{prop}\label{prop:pushpull}
    Consider the following transverse pullback square where $\iota \colon M\hookrightarrow X$ and $\kappa\colon N\hookrightarrow Y$ are embeddings of codimension $n$ of closed manifolds into the interior of compact manifolds, and $f(\partial Y)\subseteq \partial X$:
    \begin{equation}\label{eq:pushpull}
        \begin{tikzcd}
            N \arrow[d,"\kappa"',hook] \arrow[r,"g"] \arrow[rd, phantom, "\lrcorner", very near start] & M \arrow[d,"\iota",hook] \\
            Y \arrow[r,"f"] &  X
            \end{tikzcd}
    \end{equation}
    For any local coefficient system $\AA$ on $X$ and cohomology class $\alpha\in H^*(M;\iota^*\AA\otimes\mathcal{O}_{\nu\iota})$, we have
    \[
        f^*\iota^!(\alpha)=\kappa^! g^*(\alpha) \in H^{*+n}(Y,\partial Y;f^*\AA).
    \]
\end{prop}
\begin{proof}
    First note that the differential of $f$ induces a canonical isomorphism $g^*\nu\iota \cong \nu\kappa$ which induces a canonical isomorphism $g^*\mathcal{O}_{\nu\iota} \cong \mathcal{O}_{\nu\kappa}$.
    Moreover, commutativity of the diagram implies $g^*\iota^*\AA=\kappa^*f^*\AA$.
    Thus $g^*(\alpha)\in H^*(N;g^*\iota^*\AA\otimes g^*\mathcal{O}_{\nu\iota})$ may be viewed as an element of $H^*(N;\kappa^* f^*\AA\otimes\mathcal{O}_{\nu\kappa})$, and so $\kappa^! g^*(\alpha)$ is in $H^{*+n}(Y,\partial Y;f^*\AA)$ as claimed.

    Choose tubular neighbourhoods $\jmath_U\colon U\hookrightarrow X$ of $M$ and $\jmath_V\colon V \hookrightarrow Y$ of $N$, with projections \(p_U \colon U \to M\) and \(p_V \colon V \to N\), such that $f(V)\subseteq U$.
    By \cite[Theorem 4.6.7]{Hirsch}, after altering $f$ by a suitable homotopy, we may assume that $f|_V \colon V\to U$ is a map of vector bundles; then the diagram
    \[
    \begin{tikzcd}
        V \arrow[d,"p_V" swap] \arrow[r,"f|_V"] & U \arrow[d,"p_U"] \\
        N \arrow[r,"g"]                    & M
    \end{tikzcd}
    \]
    commutes and the map
    \[
        (f|_V)^*\colon H^n(U,U-M;p_U^*\mathcal{O}_{\nu\iota})\to H^n(V,V-N; (f|_V)^*p_U^*\mathcal{O}_{\nu\iota}) = H^n(V,V-N; p_V^*\mathcal{O}_{\nu\kappa})
    \]
    sends the Thom class $u_U$ to the Thom class $u_V$.
    The result now follows by commutativity of the following diagram, whose squares commute by naturality:
    \[
    \begin{tikzcd}[column sep=1.25cm, row sep=1.5em]
        H^*(N;\kappa^* f^*\AA\otimes \Or_{\nu\kappa}) \arrow[d,"p_V^*"'] & H^*(M;\iota^*\AA\otimes\Or_{\nu\iota}) \arrow[l,"g^*"] \arrow[d,"p_U^*"] \\
        H^*(V;p_V^*\kappa^* f^*\AA\otimes p_V^*\Or_{\nu\kappa}) \arrow[d,"\cup u_V"'] & H^*(U;p_U^*\iota^*\AA\otimes p_U^*\Or_{\nu\iota}) \arrow[l,"(f|_V)^*"] \arrow[d,"\cup u_U"]\\
        H^{*+n}(V,V-N ;\jmath_V^* f^*\AA) & H^{*+n}(U,U-M; \jmath_U^*\AA) \arrow[l,"(f|_V)^*"] \\
        H^{*+n}(Y,Y-N; f^*\AA) \arrow[u,"\jmath_V^*"] \arrow[d] & H^{*+n}(X,X-M; \AA) \arrow[l,"f^*"] \arrow[d] \arrow[u,"\jmath_U^*"'] \\
        H^{*+n}(Y,\partial Y;f^*\AA) & H^{*+n}(X,\partial X;\AA) \arrow[l,"f^*"]
    \end{tikzcd}
    \]
\end{proof}

The following property of the Umkehr map is well-known, see \cite{Thom52, OhmotoSaekiSakuma03} for example.

\begin{lemma}\label{lem:propUmkehr}
    Let $\iota \colon M \hookrightarrow X$ be an embedding of a closed $k$-manifold into the interior of a compact $(n+k)$-manifold.
    The twisted Umkehr map $\iota^!$ is Poincar\'e dual to $\iota_*$. More precisely, $\iota^!$ agrees with the composition
    \begin{align*}
      H^*(M;\iota^*\AA\otimes\Or_{\nu\iota}) & \xrightarrow{\cap [M]_\tw}  H_{k-*}(M;\iota^*\AA\otimes\Or_{\nu\iota}\otimes\Or_M) \\
      & \stackrel{\cong}{\longrightarrow} H_{k-*}(M;\iota^*\AA\otimes\iota^*\Or_X) \cong H_{k-*}(M;\iota^*(\AA\otimes\Or_X))\\
      & \stackrel{\iota_*}{\longrightarrow} H_{k-*}(X;\AA\otimes\Or_X) \\
      & \xrightarrow{(\cap [X,\partial X]_\tw)^{-1}} H^{*+n}(X,\partial X;\AA)  
    \end{align*}          
    where the first unnamed isomorphism comes from the isomorphisms of local coefficient systems
    \[
    \Or_{\nu\iota}\otimes\Or_M\cong \Or_{\nu\iota}\otimes\Or_{TM}\cong \Or_{\nu\iota\oplus TM} \cong \Or_{\iota^*TX}\cong \iota^*\Or_{TX}\cong \iota^*\Or_X
    \]
    furnished by \textup{Lemma \ref{lemma:prop_loc_sys}}.
    Here $[M]_\tw\in H_k(M;\Or_M)$ and $[X,\partial X]_\tw\in H_{n+k}(X,\partial X;\Or_X)$ are twisted fundamental classes.
\end{lemma}

\subsection{The twisted Thom space and Thom class}
\label{subsec:twisted_thom_space_class}

Recall that the determinant induces a map on classifying spaces $w_1 \colon \BO(n) \to \BO(1)$, which we may choose to be a fibration.
In this section, we use the map $w_1$ to obtain a twisted Thom space parametrised over \(\BO(1)\).

\begin{definition}\label{defn:twThom_pushout}
    Let $\pi_{\gamma_n} \colon \gamma_n \to \BO(n)$ be the universal $n$-plane bundle and denote by $D(\gamma_n)$ its disc bundle and by $S(\gamma_n)$ its sphere bundle.
    The space $\twThom(n)$ is defined as the pushout
        \begin{equation}\label{eq:twThom}
        \begin{tikzcd}[sep=large]
            S(\gamma_n) \arrow[r, "w_1 \circ \pi_{S(\gamma_n)}"] \arrow[d, hook] & \BO(1) \arrow[d, "s"] \\
            D(\gamma_n) \arrow[r, "r"] & \twThom(n) \mathrlap{,} \arrow[ul, phantom, "\ulcorner", very near start]
        \end{tikzcd}
    \end{equation}
    and is called the \emph{twisted universal Thom space for $\OO(n)$}.
\end{definition}

By Lemma~\ref{lem:fibrewisepushouts}, we obtain a map $q \colon \twThom(n) \to \BO(1)$ so that the triple $(\twThom(n); q, s)$ is a retractive space in $\mathcal{K}^*_{/\BO(1)}$.
Moreover, since we chose $w_1 \colon \BO(n) \to \BO(1)$ to be a fibration, the retractive space $(\twThom(n), q, s)$ is fibrant.

\begin{remark}
    Since $S(\gamma_n) \hookrightarrow D(\gamma_n)$ is a cofibration, the diagram~\eqref{eq:twThom} is a homotopy pushout.
\end{remark}

\begin{remark}\label{rmk:model_bo_homotopy_orbits}
    Let us see that the homotopy quotient of $\BSO(n)$ is a model for $\BO(n)$.
    Given models for $\EO(n)$ and $\E\ZZ_2$, consider the diagonal action of $\OO(n)$ on the product $\E\ZZ_2\times\EO(n)$, where the action on the first factor is given by the determinant map $\OO(n)\to \ZZ_2$. This is a free action on a contractible $\OO(n)$-CW complex. Furthermore, the inherited action of the index $2$ (normal) subgroup $\SO(n)$ is also free (and trivial on the first factor), so
    \[
        \left.\big(\E\ZZ_2 \times \EO(n)\big) \middle/\SO(n)\right.=\E\ZZ_2\times \BSO(n)
    \]
    is a model for \(\BSO(n)\).
    We compute the homotopy quotient of \(\BSO(n)\) with respect to the residual action of \(\OO(n)/\SO(n) \cong \ZZ_2\) as follows: 
    \begin{align*}
        \BSO(n)\sslash\ZZ_2 & = \E\ZZ_2\times_{\ZZ_2} \BSO(n) \\
                            & \cong \left. \left(\sfrac{(\E\ZZ_2 \times \EO(n))}{\SO(n)}\right) \middle/ \left(\sfrac{\OO(n)}{\SO(n)}\right) \right. \\
                            & \cong \sfrac{(\E\ZZ_2\times\EO(n))}{\OO(n)}\\
                            & \cong \BO(n).
    \end{align*}

    In a similar fashion, we can take as a model for the universal rank $n$ vector bundle $\gamma_n\to \BO(n)$ the $\ZZ_2$-homotopy orbits of a universal oriented $n$-plane bundle $\tilde{\gamma}_n \to \BSO(n)$. Hence we can realise the twisted Thom space as the following pushout:
    \begin{equation}\label{eq:twThom_as_homotopy_orbits}
    \begin{tikzcd}
        \mathllap{S(\tilde{\gamma}_n) \sslash \ZZ_2 =\;} S(\tilde{\gamma}_n \sslash \ZZ_2) \arrow[d, hook] \arrow[r] &  \B\ZZ_2\arrow[d, "s"] \\
        \mathllap{D(\tilde{\gamma}_n) \sslash \ZZ_2 =\;} D(\tilde{\gamma}_n \sslash \ZZ_2)  \arrow[r] & \twThom(n) \arrow[ul, phantom, "\ulcorner", very near start] \mathrlap{.}
    \end{tikzcd}
    \end{equation}
    Recall that we can write $\MSO(n)$ as the pushout
    \begin{equation}\label{eq:MSO(n)_pushout}
    \begin{tikzcd}
        S(\tilde{\gamma}_n) \arrow[r] \arrow[d, hook] & \ast \arrow[d] \\
        D(\tilde{\gamma}_n)  \arrow[r] &   \MSO(n) \arrow[ul, phantom, "\ulcorner", very near start]
    \end{tikzcd}
    \end{equation}
    in the category $\mathcal{K}_{\ZZ_2}$.
    According to Lemma \ref{lem:propsHQ}\,(i), applying homotopy quotients to \eqref{eq:MSO(n)_pushout} recovers precisely the pushout~\eqref{eq:twThom_as_homotopy_orbits} in the category $\mathcal{K}$.
    In other words, under the identifications above, we can write
    \[
        \twThom(n) = \MSO(n) \sslash \ZZ_2,
    \]
    an identification we will use freely when the need arises.
\end{remark}

Returning to the task at hand, we next establish some notational shortcuts.
By construction (cf.\ Definition \ref{defn:twThom_pushout}), the twisted Thom space is homeomorphic to the adjunction space
\[
    \twThom(n) = D(\gamma_n) \bigcup_{w_1 \circ \pi_{S(\gamma_n)}} \BO(1).
\]
As such, we may regard the open disc bundle $\mathring{D}(\gamma_n)$ and $\BO(1)$ as parametrised subspaces of $\twThom(n)$, and thus freely omit the maps $r$ and $s$ of the pushout square \eqref{eq:twThom} from our notation as we see fit.
Moreover, as discussed in Remark \ref{rmk:universal_orientation_system},  the classifying space \(\BO(1)\) is equipped with the universal orientation system $\Or_{\gamma_1}$ induced by the universal line bundle $\gamma_1 \to \BO(1)$.
This coefficient system pulls back to $\twThom(n)$, which we define as
\[
    \mathcal{L}_\ZZ^{\mathrm{Th}} \coloneq q^*\Or_{\gamma_1} \cong \Or_{q^*\gamma_1}.
\]

We conclude this subsection by introducing the universal twisted Thom class on $\twThom(n)$.
Recall from Subsection~\ref{subsec:twistedThomisos} that the universal $n$-plane bundle $\gamma_n$ admits a canonical Thom class
\[
    u_{\gamma_n} \in H^n(\gamma_n,\gamma_n-z(\BO(n));\pi_{\gamma_n}^*\Or_{\gamma_n})\cong H^n(D(\gamma_n),S(\gamma_n);\pi_{\gamma_n}^*\Or_{\gamma_n}).
\]
Now observe that $\Or_{\gamma_n} \cong \Or_{\det\gamma_n}\cong w_1^*\Or_{\gamma_1}$ by Lemma~\ref{lemma:prop_loc_sys}, and therefore
\[
\pi_{\gamma_n}^*\Or_{\gamma_n}\cong \pi_{\gamma_n}^*w_1^*\Or_{\gamma_1}= r^*\mathcal{L}_\ZZ^\mathrm{Th}.
\]
Since the map $S(\gamma_n) \hookrightarrow D(\gamma_n)$ in the diagram \eqref{eq:twThom} is a cofibration, the triad
\[
    \left( \twThom(n); D(\gamma_n), \BO(1) \right)
\]
is excisive, and there is an excision isomorphism
\begin{align*}
    r^* \colon H^n(\twThom(n), \BO(1); \mathcal{L}_\ZZ^\mathrm{Th})
    &\stackrel{\cong}{\longrightarrow} 
    H^n(D(\gamma_n), S(\gamma_n); r^*\mathcal{L}_\ZZ^\mathrm{Th}), \\
    &\cong
    H^n(D(\gamma_n), S(\gamma_n); \pi_{\gamma_n}^* \Or_{\gamma_n}),
\end{align*}
see~\cite[p.\ 271]{Whitehead79} for excision in cohomology with local coefficients.

We define the \emph{universal twisted Thom class} $u_n^\tw\in H^n(\twThom(n),\BO(1);\mathcal{L}_\ZZ^{\mathrm{Th}})$ to be the class corresponding to $u_{\gamma_n}$ under this isomorphism.

\begin{remark}
    Even though $\gamma_n$ admits a preferred twisted Thom class
    \[
        u_{\gamma_n} \in H^n(D(\gamma_n),S(\gamma_n);\pi_{\gamma_n}^*\Or_{\gamma_n})\stackrel{\eqref{eq:twisted_thom_iso_I}}{\cong} H^0(\BO(n);\ZZ)\cong\ZZ
    \]
    there are, in general, exactly two choices of generators and hence twisted Thom classes which differ by a sign.
    Passing to the double cover associated with $\Or_{\gamma_n}$, in other words, pulling back along the map of pairs
    \[
        (D(\tilde{\gamma}_n),S(\tilde\gamma_n))\to (D(\gamma_n),S(\gamma_n)),
    \]
    these correspond to the two universal oriented Thom classes in
    \[
        H^n(D(\tilde{\gamma}_n),S(\tilde\gamma_n);\ZZ)\cong H^0(\BSO(n);\ZZ)\cong \ZZ.
    \]
\end{remark}
\section{Twisted Pontryagin--Thom construction and realisability}\label{sec:realisability}

In this section, we introduce the concept of cobordisms twisted by a line bundle and discuss the corresponding twisted Pontryagin--Thom construction.

\subsection{Parametrised transversality}
\label{subsec:parametrised_transversality}

Denote by $\catname{Mfd}$ the category of smooth manifolds without boundary.
For a fixed $B \in \catname{Mfd}$, we denote the category of manifolds over~$B$ by $\catname{Mfd}_{/B}$.
Transversality in the category $\catname{Mfd}_{/B}$ is less robust than in the category $\catname{Mfd}$.
For example, let $(Y; p_Y) \in \catname{Mfd}_{/B}$ and assume that $p_Y \colon Y \to B$ is a fibre bundle.
If a submanifold $A \subset Y$ does not meet the fibres of $p_Y$ transversely, and $A$ has relatively low codimension with respect to the fibres of $p_Y$, then there is in general no way we can perturb a given parametrised map $f \colon (X; p_X) \to (Y; p_Y)$ through parametrised maps in $\catname{Mfd}_{/B}$ to achieve transversality with respect to $A$.
However, if there exists a parametrised tubular neighbourhood, the next lemma demonstrates that we can always achieve transversality independent of whether $p_Y$ is a fibre bundle or not. 

\begin{lemma}\label{lem:fibrewise_transversality}
    Let $f \colon (X; p_X) \to (Y; p_Y)$ be a parametrised map in $\catname{Mfd}_{/B}$, and let $A \subset Y$ be a closed, compact submanifold.
    Assume that $A$ admits a tubular neighbourhood $\psi \colon U \xrightarrow{\cong} \nu A$ for which $\pi_{\nu A} \circ \psi \colon U \to A$ is a parametrised map over $B$ via $p_Y$.
    Then $f$ is (smoothly) parametrised homotopic to a smooth, parametrised map that is transverse to $A$.
\end{lemma}
\begin{proof}
    Denote by $\tilde{U} \coloneq f^{-1}(U)$ and let $V \subset X$ be an open neighbourhood of $f^{-1}(A)$ for which $\overline{V} \subset \tilde{U}$.
    Choose a smooth function $\varphi \colon X \to [0,1]$ with $\restr{\varphi}{V} \equiv 1$ and $\restr{\varphi}{X - \tilde{U}} \equiv 0$.
    Let $(s_1, \dots, s_N)$ be sections of $\nu A$ so that
    \begin{align*}
        s_1(x), \dots, s_N(x) \quad \text{spans $\nu_x A$ for all $x \in A$.}
    \end{align*}
    Note that we can find finitely many of those sections because $A$ is compact.
    For $\alpha \in \RR^N$, we denote $$L_\alpha \coloneq \sum_{i=1}^N \alpha_i s_i.$$
    Now, we want to push $f$ away from $A$ in every normal direction $s_1, \dots, s_N$ through defining a smooth map $F \colon X \times \RR^N \to Y$ given by
    \begin{align*}
        F(x, \alpha) = \begin{cases}
            \psi^{-1}\bigl( \varphi(x) L_\alpha(\pi_{\nu A} \circ \psi \circ f(x)) + \psi(f(x)) \bigr) & \text{if $x \in \tilde{U}$,} \\
            f(x) & \text{if $x \notin \tilde{U}$}.
        \end{cases}
    \end{align*}
    Note that $F$ is a parametrised map because for $x \in \tilde{U}$, we have:
    \begin{align*}
        \left(p_Y \circ F\right)(x, \alpha) &= \left(p_Y \circ \psi^{-1} \right) \bigl( \varphi(x) L_\alpha(\pi_{\nu A} \circ \psi \circ f(x)) + \psi(f(x)) \bigr) \\
            &= \left(p_Y \circ \pi_{\nu A} \right) \bigl( \varphi(x) L_\alpha(\pi_{\nu A} \circ \psi \circ f(x)) + \psi(f(x)) \bigr) \\
            &= \left( p_Y \circ \pi_{\nu A} \circ \psi \right)(f(x)) \\
            &= \left(p_Y \circ f \right)(x) \\
            &= p_X(x).
    \end{align*}
    Here we used $p_Y \circ \pi_{\nu A} \circ \psi = \restr{p_Y}{U}$ because $\pi_{\nu A} \circ \psi$ is parametrised by assumption.
    By construction, $F$ is transverse to $A$.
    Therefore, by the Parametric Transversality Theorem, there exists an $\alpha_0 \in \RR^N$ for which $F(\placeholder, \alpha_0)$ is transverse to $A$.
    The desired parametrised homotopy is now given by $(x,t) \mapsto F(x, t \alpha_0)$.
\end{proof}

\begin{remark}
    In Lemma~\ref{lem:fibrewise_transversality}, a sufficient condition for $A \subset Y$ to admit such a parametrised tubular neighbourhood is that $\restr{p_Y}{A}$ is a submersion.
\end{remark}

\subsection{Twisted cobordisms}\label{subsec:twisted_cobordisms}

Let $X$ be a closed $(n+k)$-dimensional manifold and $\xi \to X$ a smooth line bundle.

\begin{definition}
    A \emph{$\xi$-oriented submanifold of $X^{n+k}$} of dimension $k$ is a tuple $(M, \varphi)$ where $M \subset X$ is a closed embedded submanifold of dimension $k$ and
    \[
        \varphi \colon \bwedge^n\! \nu M \longrightarrow \restr{\xi}{M}
    \]
    is a bundle isomorphism.
\end{definition}

At last, we define the notion of \(\xi\)-oriented cobordism, which generalises that of oriented cobordism. We preemptively flag that the lack of orientation reversal of a boundary component in Definition \ref{defn:twisted_L_equivalence}\,(iii) occurs because we are concerned with orienting normal bundles as opposed to tangent bundles; for more clarification about the relationship between Definitions \ref{def:oriented_cobordism}\,(iii) and \ref{defn:twisted_L_equivalence}\,(iii), see Remark \ref{rem:oriented_vs_twisted_cobordism} below.

Denote by $\pr_X \colon X \times [0,1] \to X$ the projection onto $X$.

\begin{definition}\label{defn:twisted_L_equivalence}
    We say that two closed $\xi$-oriented $k$-dimensional submanifolds $(M_i, \varphi_i)$ with $i=0,1$, are \emph{$\xi$-oriented cobordant} if there exists a $(\pr_X^*\xi)$-oriented compact $(k+1)$-dimensional submanifold $(W,\varphi)$ of $X \times [0,1]$ such that
    \begin{enumerate}[(i)]
        \item $W$ intersects $X \times \{i\}$ transversely;
        \item $\partial W = M_0 \times \{0\} \sqcup M_1 \times \{1\}$; and
        \item $\varphi$ restricts to $\varphi_i$ on $M_i \times \{i\}$ for $i=0,1$.
    \end{enumerate}
    We denote the set of $\xi$-oriented cobordism classes of dimension $k$ by $\LL_k(X; \Or_X \otimes \Or_\xi)$.
\end{definition}

Our choice of the notation $\LL_k(X; \Or_X \otimes \Or_\xi)$ will become apparent later.

\begin{remark}\label{rem:oriented_vs_twisted_cobordism}
    Note that Definition~\ref{defn:twisted_L_equivalence} contains Definition~\ref{def:oriented_cobordism} by choosing $\xi$ to be the trivial line bundle over an oriented $X$ since a choice of isomorphism \( \Or_{\nu M} \cong \underline{\ZZ} \) amounts precisely to a choice of orientation of the normal bundle $\nu M$ for a submanifold $M \subset X$, which in turn, since $X$ is oriented, determines an orientation of $M$.
\end{remark}

\subsection{A manifold model for the twisted Thom space}\label{subsec:manifold_model_for_thom_space}

In this section, we want to fix models for $\BO(n)$ and $\BO(1)$ that can not only be approximated by manifolds, but also (as necessitated by our application) concretise $(\twThom(n), q, s)$ as a fibrant retractive space.
To this end, one might be tempted to model $\BO(n)$ by the Grassmannians $\Gr_n(\RR^\infty)$ of $n$-planes in $\RR^\infty$ using the Pl\"ucker embedding:
\begin{align*}
    \iota_P \colon \Gr_n(\RR^N) \hookrightarrow P\bigl(\bwedge^n \RR^N \bigr), \qquad \operatorname{span}(v_1, \dots v_n) \mapsto [v_1 \wedge \dots \wedge v_n].
\end{align*}
This is fruitless, however, for the colimit of \(\iota_P\) does not realise $(\twThom(n), q, s)$ as a fibrant space over $\RR P^\infty$.
Subsequently, we will instead build our model upon homotopy orbits of $\BSO(n)$ as discussed in Remark~\ref{rmk:model_bo_homotopy_orbits}.

For $N >0$, let $\tilde{\gamma}_{n,N} \to \widetilde{\Gr}_n(\RR^N)$ be the canonical oriented $n$-plane bundle over the oriented Grassmannians $\widetilde{\Gr}_n(\RR^N)$ of oriented $n$-planes in $\RR^N$.
Both the space $\widetilde{\Gr}_n(\RR^N)$ and the bundle $\tilde{\gamma}_{n,N}$ are equipped with a $\ZZ_2$-action given by reversing the orientation of a given oriented $n$-plane.
Then, for each $L,N >0$, there is a fibre bundle
\begin{align*}
    \Gr_{n,N}^{L} \coloneq \widetilde{\Gr}_n(\RR^N) \times_{\ZZ_2} S^L \to \RR P^L,
\end{align*}
where $S^L$ is equipped with the antipodal action and $S^L \to \RR P^L$ is given by the corresponding quotient map.
The bundle $\tilde{\gamma}_{n,N}$ induces an $n$-plane bundle
\begin{align*}
    \gamma_{n,N}^L \coloneq \tilde{\gamma}_{n,N} \times_{\ZZ_2} S^L \to \Gr_{n,N}^{L}.
\end{align*}
Note that $\gamma_{n,N}^L \to \RR P^L$ is also a fibre bundle with fibre $\tilde{\gamma}_{n,N}$.
We then model $\BO(n)$ by $$\Gr_{n,\infty}^\infty \coloneq \colim_{L,N \to \infty} \Gr_{n,N}^{L} = \widetilde{\Gr}_n(\RR^\infty) \times_{\ZZ_2} S^\infty,$$ which is equivalent to the homotopy orbits of $\widetilde{\Gr}_n(\RR^\infty)$ and therefore a fibre bundle over~$\RR P^\infty$.
Moreover, the vector bundle $\colim_{L,N}\gamma_{n,N}^L \eqqcolon \gamma_n \to \Gr_{n,\infty}^{\infty}$ becomes the classifying $n$-plane bundle over it.
Again, the bundle $\gamma_n \to \RR P^\infty$ is also a fibre bundle with fibre $\tilde{\gamma}_n$, that is, the canonical oriented $n$-plane bundle over $\widetilde{\Gr}_n(\RR^\infty)$.
We define the following pushout:
\[
\begin{tikzcd}
    \mathllap{S(\tilde{\gamma}_{n,N}) \times_{\ZZ_2} S^L = \;} S(\gamma^L_{n,N}) \arrow[r] \arrow[d, hook] & \RR P^L \arrow[d] \\
    \mathllap{D(\tilde{\gamma}_{n,N}) \times_{\ZZ_2} S^L = \;} D(\gamma^L_{n,N}) \arrow[r] & \mathrm{M}^{(L,N)}_\tw\mathrm{O}(n) \mathrlap{.} \arrow[ul, phantom, "\ulcorner", very near start]
\end{tikzcd}
\]
By construction, we have $\colim_{(L,N)} \mathrm{M}_{(L,N)}^\tw\mathrm{O}(n) = \twThom(n)$.
The map $q \colon \twThom(n) \to \RR P^\infty$ is induced by the projection $D(\gamma^\infty_{n,\infty}) \to \RR P^\infty$, which is a fibre bundle hence a fibration. By an application of Lemma~\ref{lem:fibrewisepushouts} we obtain a fibrant retractive space over $\RR P^\infty$.

\subsection{The twisted Pontryagin--Thom construction}
\label{subsec:Pontryagin}

Throughout this section, we consider a smooth map \(p \colon X \to \RR P^\infty,\) where \(X\) is a closed $(n+k)$-dimensional manifold.
By smoothness of the map $p$, we mean the following.
Since \(X\) is compact, there always exists a number \(L >0\) for which $\operatorname{im}(p) \subset \RR P^L$, so we require that $p$ is smooth after restricting its codomain to $\RR P^L$.
Altogether, the tuple $(X,p)$ assembles to a parametrised space over $\RR P^\infty$, and we write
\[
    \xi \coloneq p^* \gamma_1
\]
for the induced line bundle over $X$. We emphasise that any line bundle \(\xi\) appearing in this section must come from this map into $\RR P^\infty$.

\begin{construction}[Twisted Pontryagin manifold]\label{construction:pontryagin_manifolds}
    Given a parametrised smooth manifold \((X,p)\) with $p \colon X \to \RR P^\infty$ as above, choose a representative $f \colon X \to \twThom(n)$ of a parametrised homotopy class in $[X, \twThom(n)]_{\RR P^\infty}$. Using again the compactness of $X$, note that
    \[
        \operatorname{im}(f) \subset \mathrm{M}_{(L,N)}^\tw\mathrm{O}(n) \qquad \text{for sufficiently large \(L \text{ and } N.\)}
    \]
    We may assume without loss of generality that $f$ is smooth in a neighbourhood $U$ of the zero section $z \colon \Gr_{n,N}^L \hookrightarrow \mathring{D}(\gamma_{n,N}^L)$. Indeed, if we write $V \coloneq f^{-1}(U)$, and \(f\) is not already smooth on \(V\), then, regarding $\mathring{D}(\gamma^L_{n,N})$ as a fibre bundle over $\RR P^L$, the restriction $\restr{f}{V}$ is a smoothable section of the pullback bundle $\restr{p}{U}^*\mathring{D}(\gamma^{L}_{n,N})$. Additionally, according to Lemma~\ref{lem:fibrewise_transversality}, we may assume that $f$ intersects the zero section $z$ transversely.
    Altogether, having made all necessary assumptions, we have that $M_f\coloneq f^{-1}(z)$ is a closed submanifold of $X$ with codimension~$n$.
    The canonical identification of the normal bundle of $\Gr_{n,N}^L$ with \(\gamma_{n,N}^L = \restr{\gamma_n}{\Gr_{n,N}^L}\) and the identity $q \circ f =p$ together yield a canonical bundle isomorphism
    \begin{align*}
        \bwedge^n\!\restr{f}{M_f}^*\!\gamma_n &\cong \restr{f}{M_f}^* \bigl(\bwedge^n\! \gamma_n \bigr) \\
            &\cong \restr{f}{M_f}^* (q^* \gamma_1) \\
            &= \bigl(q \circ \restr{f}{M_f}\bigr)^*\!\gamma_1 \\
            &= \restr{p}{M_f}^*\!\gamma_1 \\
            &= \restr{\xi}{M_f}.
    \end{align*}
    Subsequently, the map $\bwedge^n\! \mathrm{d}f$ induces an isomorphism $\varphi_f \colon \bwedge^n\! \nu {M_f} \xrightarrow{\cong} \restr{\xi}{M_f}$.
    We call the tuple $(M_f, \varphi_f)$ the \emph{twisted Pontryagin manifold of~$f$}.
\end{construction}
Directly, Construction~\ref{construction:pontryagin_manifolds} gives rise to a map
\begin{equation}\label{eq:d}
    d \colon [X, \twThom(n)]_{\RR P^\infty} \longrightarrow \LL_k(X; \Or_X \otimes \Or_\xi), \qquad [f]_{\RR P^\infty} \longmapsto \left[M_f, \varphi_f \right].
\end{equation}
To see that this map is well-defined, assume that two parametrised maps $f_i \colon X \to \twThom(n)$, for $i=0,1$ with the above properties, are parametrised homotopic through some parametrised homotopy $h \colon X \times [0,1] \to \twThom(n)$.
Without loss of generality, we may assume that $h$ is constant near the endpoints, smooth in $\mathring{D}(\gamma^L_{n,N})$, and transverse to the zero section $z$ away from the endpoints.
We then set $W_h \coloneq h^{-1}(z)$, and the twisted Pontryagin manifold $(W_h, \varphi_h)$ associated with $h$ is the desired $\xi$-oriented cobordism between the twisted Pontryagin manifolds of the parametrised homotopic maps $f_0$ and $f_1$.

We will now describe the twisted collapse map
\[
    c \colon \LL_k(X; \Or_X \otimes \Or_\xi) \longrightarrow [X, \twThom(n)]_{\RR P^\infty}.
\]
We will show later that $c$ is inverse to $d$.
We remark that the discussion in Cruickshank~\cite[Section 5.1]{Cruickshank03} inspired the following construction that is essential to define this map.

\begin{construction}\label{construction:collapse}
Suppose that $(M, \varphi)$ is a $\xi$-oriented submanifold of $X$.
Fix a closed tubular neighbourhood $\psi \colon T \to D(\nu M)$ of $M$.
We define the following vector bundle $E \to T$ via the pullback
\[
\begin{tikzcd}[column sep=large]
    E \arrow[r, "\rho"] \arrow[d] \arrow[rd, phantom, "\lrcorner", very near start] & \nu M \arrow[d, "\pi_{\nu M}"] \\
    T \arrow[r, "\sim", "\pi_{\nu M} \circ \psi"'] & M \mathrlap{.}
\end{tikzcd}
\]
Since $\pi_{\nu M} \circ \psi$ is a strong deformation retract onto $M$, this bundle map is unique up to vertical homotopy, and we have $\restr{E}{M} = \nu M$. Next observe that there is an isomorphism $\varphi\colon \bwedge^n \nu M \to \restr{\xi}{M}$ that extends to an isomorphism $\varphi_T\colon \bwedge^n E \to \restr{\xi}{T}$; therefore, there is a classifying map $g \colon T \to \Gr_{n,\infty}^\infty$ for the bundle $E$ that is a lift of $\restr{p}{T} \colon T \to \RR P^\infty$ through the fibration $w_1 \colon \Gr_{n,\infty}^\infty\to \RR P^\infty$. Next, as indicated in diagram \eqref{eq:Gfactors} below, we may choose the bundle map $G \colon E \to \gamma_n$ covering~$g$ so that the exterior power $\bwedge^n G\colon \bwedge^n E\to \bwedge^n \gamma_n$ factors the given $\xi$-orientation $\varphi$ over $M$:
\begin{equation}\label{eq:Gfactors}
\begin{tikzcd}
    \bwedge^n \nu M \arrow[r, equal]  \arrow[d, "\varphi"'] & \bwedge^n \restr{E}{M} \arrow[r, "\bwedge^n G"] & \bwedge^n \gamma_n \arrow[d] \\
    \restr{\xi}{M} \arrow[r, equal] & \restr{p}{M}^* \gamma_1 \arrow[r] & \gamma_1 \mathrlap{.}
\end{tikzcd}
\end{equation}

We now identify $D(E)$ with the parametrised product $D(\nu M) \times_M D(\nu M)$ using the following pullback diagram:
\begin{equation}\label{eq:identify_diagonal_pullback}
\begin{tikzcd}
    D(E) \arrow[rd, "\overline{\psi}", dashed] \arrow[rrd, "\rho"', bend left=25] \arrow[d, "\pi_E"'] &                                                                         &                                 \\
    T \arrow[rd, "\psi"']                            & D(\nu M) \times_M D(\nu M) \arrow[r, "\beta"] \arrow[d, "\alpha"] & D(\nu M) \arrow[d, "\pi_{\nu M}"] \\
                                                                                      & D(\nu M) \arrow[r, "\pi_{\nu M}"']                                         & M \mathrlap{.}
\end{tikzcd}
\end{equation}
By way of explanation, the maps $\alpha$ and $\beta$ are the projections onto the first and second component, respectively; furthermore, since the lower right square is a pullback square, the maps $\rho$ and $\psi $ induce a map
\[
    \overline{\psi} \colon D(E) \to D(\nu M) \times_M D(\nu M);
\]
and, lastly, it is easy to see that $\overline{\psi}$ is a bundle map that covers~$\psi$, whereby $\overline{\psi}$ is particularly a diffeomorphism.
Notice that $\alpha$ has a section $\Delta \colon D(\nu M) \to D(\nu M) \times_M D(\nu M)$ given by the diagonal map $v \mapsto (v,v)$. We define $\sigma \coloneq \overline{\psi}^{-1} \circ \Delta \circ \psi \colon T \to D(E)$ and compute
\begin{align*}
    \pi_E \circ \sigma &= \pi_{E} \circ \overline{\psi}^{-1} \circ \Delta \circ \psi \\
        &= \psi^{-1} \circ \alpha \circ \Delta \circ \psi \\
        &= \psi^{-1} \circ \id_{\nu M}\circ \psi \\
        &= \id_T,
\end{align*}
from which we conclude that the map $\sigma$ is a section of $\pi_E$ on $T$.

Finally, we define the map $f_M \colon X \to \twThom(n)$ as follows:
\begin{align}\label{eq:collapse_map_construction}
    f_M(x) = \begin{cases}
        \left(G \circ \sigma \right)(x) & \text{if $x \in \mathring{T}$,} \\
        p(x) & \text{if $x \notin\mathring{T}$.}
    \end{cases}
\end{align}
It remains to check that $f_M$ is continuous on $\twThom(n)$ and a parametrised map.
To see this, observe that
\begin{align*}
    \left(w_1 \circ \pi_{\gamma_n} \circ G \circ \sigma \right)(x) &= \left(w_1 \circ g \circ \pi_E \circ \sigma \right)(x) \\
        &= \left(w_1 \circ g \right)(x) \\
        &= p(x),
\end{align*}
and recall that $\twThom(n)$ is given by the adjunction space $D(\gamma_n) \cup_{w_1 \circ \pi_{S(\gamma_n)}} \RR P^\infty$.
Therefore, $f_M$ is continuous and a parametrised map since $q \circ f_M = p$.
An illustration of this construction is provided in Figure~\ref{fig:twisted_pontryagin_thom}.
\end{construction}

\begin{figure}[htb]
    \centering
    \begin{overpic}[width=\textwidth]{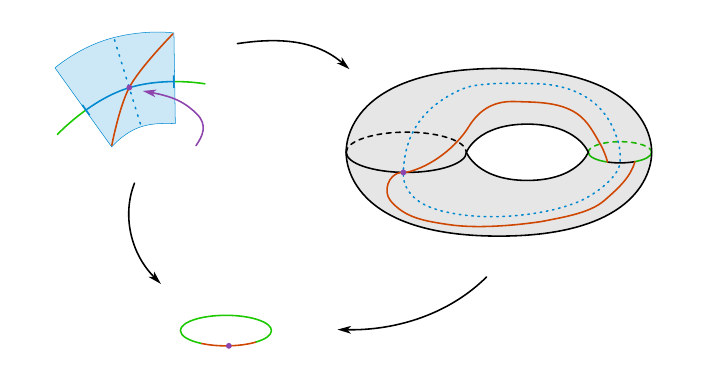}
        \put (15,20) {$p$}
        \put (55,10) {$q$}
        \put (40,49) {$f_M$}
        \put (31,40) {$X$}
        \put (24,30) {\color{vpurple} $M$}
        \put (4,48) {\color{vblue} $D(E)$}
        \put (41,31) {\color{black!75} $\Gr_{n,\infty}^\infty$}
        \put (93,31) {\color{black!75} $\RR P^\infty$}
        \put (70,46) {$\twThom(n)$}
        \put (17,6) {$\RR P^\infty$}
    \end{overpic}
    \caption{The twisted Pontryagin--Thom construction to define the collapse map. We identify $D(E)$ (blue) with the parametrised product $D(\nu M) \times_M D(\nu M)$. Then we use the diagonal map $D(\nu M) \to D(\nu M) \times_M D(\nu M)$ (red) to construct a parametrised map into $\twThom(n)$ over the tubular neighbourhood $T$ and extend it onto $X$ using the classifying map $p$ of the line bundle $\xi$ (green).}
    \label{fig:twisted_pontryagin_thom}
\end{figure}
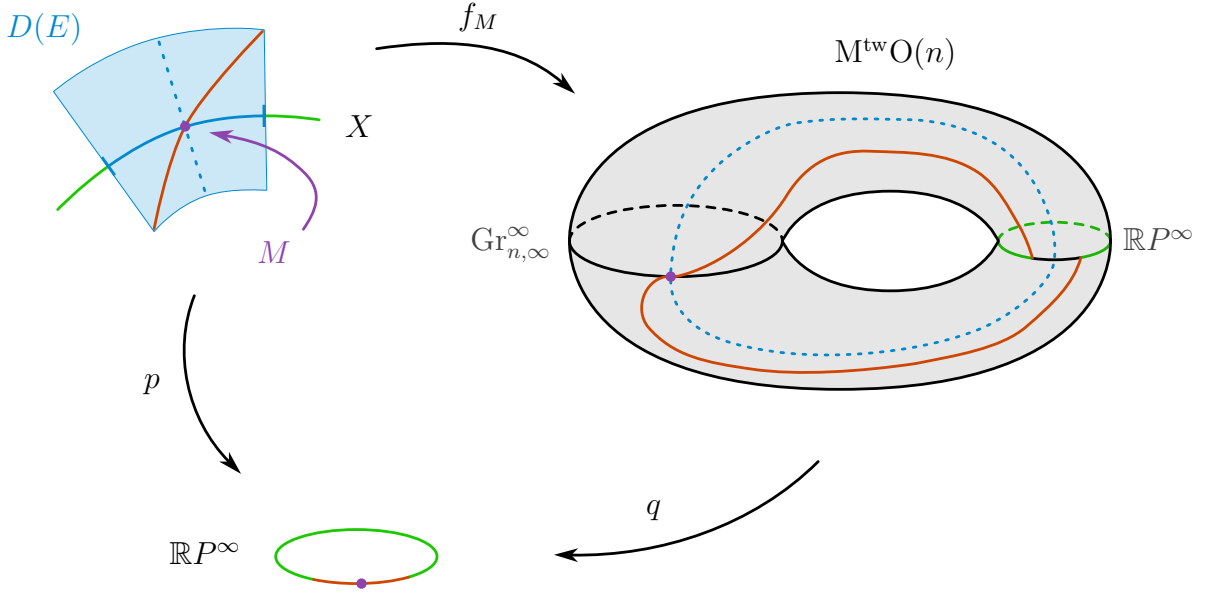

\begin{remark}
    With the explicit model of $\BO(n)$ as the homotopy orbits of $\BSO(n)$, there is another way to think of the bundle map $G$ in Construction~\ref{construction:collapse}.
    Denote by $\widetilde{T}\to T$ the double cover induced by $\restr{\xi}{T}$, and by $\widetilde{E}$ the pullback of $E$ to $\widetilde{T}$.
    Then $\restr{\xi}{T}$ pulls back to a canonically oriented line bundle $\restr{\xi}{\tilde T}$, and the isomorphism $\varphi_T\colon \bwedge^n E \to \restr{\xi}{T}$ pulls back to an isomorphism $\bwedge^n \widetilde{E} \to \restr{\xi}{\tilde T}$. Thus $\widetilde{E}$ is oriented, and can be classified by an orientation-preserving bundle map
    \begin{equation}\label{eq:Gtilde}
    \begin{tikzcd}
        \widetilde{E} \arrow[r, "\widetilde{G}"] \arrow[d] & \tilde{\gamma}_n \arrow[d] \\
        \widetilde{T} \arrow[r, "\tilde{g}"] & \BSO(n) \mathrlap{.}
    \end{tikzcd}
    \end{equation}
    Now the bundle map in \eqref{eq:Gtilde}, together with the classifying maps $\widetilde{E}\to \E\ZZ_2$ and $\widetilde{T}\to\E \ZZ_2$, descends to the homotopy orbits, which yields a bundle map
    \[
    \begin{tikzcd}
        E \arrow[r, "G"] \arrow[d] & \tilde{\gamma}_n \sslash \ZZ_2 \arrow[d] \\
        T \arrow[r, "g"] & \BSO(n) \sslash \ZZ_2
    \end{tikzcd}
    \]
    satisfying the desired properties.
\end{remark}

Ultimately, Construction \ref{construction:collapse} allows us to define the following twisted collapse map:
\begin{equation}\label{eq:c}
    c \colon \LL_k(X; \Or_X \otimes \Or_\xi) \longrightarrow [X, \twThom(n)]_{\RR P^\infty}, \qquad [(M, \varphi)] \longmapsto [f_M]_{\RR P^\infty}.
\end{equation}
To show that this map is well-defined, observe that the parametrised homotopy class of $f_M$ does not depend on the choice of tubular neighbourhood $\psi$, the lift $g$, or the bundle isomorphism $\varphi_T$.
Finally, if $(M_i, \varphi_i)$ for $i=0,1$ are $\xi$-cobordant through a cobordism $(W, \varphi)$, applying Construction~\ref{construction:collapse} to the cobordism yields a parametrised homotopy between $f_{M_0}$ and $f_{M_1}$, as required.

Combining the above Constructions \ref{construction:pontryagin_manifolds} and \ref{construction:collapse}, we can now state and prove a twisted analogue of the classical Pontryagin--Thom isomorphism.

\begin{theorem}[Twisted Pontryagin--Thom isomorphism]\label{thm:twistedPT}
    Let $p \colon X \to \RR P^\infty$ be a (smooth) map and $\xi = p^*\gamma_1$.
    Then the map
    \[
        d \colon [X, \twThom(n)]_{\RR P^\infty} \to \LL_k(X; \Or_X \otimes \Or_\xi),
    \]
    which, as defined in \textup{\eqref{eq:d}}, is induced by assigning to a parametrised map its twisted Pontryagin manifold, is inverse to the collapse map
    \[
        c \colon \LL_k(X; \Or_X \otimes \Or_\xi) \to [X, \twThom(n)]_{\RR P^\infty},
    \]
    as defined in \textup{\eqref{eq:c}}. 
    In particular, there is an isomorphism
    \[
        [X, \twThom(n)]_{\RR P^\infty} \cong \LL_k(X; \Or_X \otimes \Or_\xi).
    \]
\end{theorem}
\begin{proof}
    First, we prove that the composition $c \circ d$ is the identity on $[X, \twThom(n)]_{\RR P^\infty}$.
    In other words, we need to show that a parametrised map $g \colon X \to \twThom(n)$ with associated twisted Pontryagin manifold $(M_g, \varphi_g)$ is parametrised homotopic to $f_{M_g}$ as constructed in~\eqref{eq:collapse_map_construction}.
    Towards this, let $\psi \colon U \to \mathring{D}(\nu M_g)$ be a tubular neighbourhood of $M_g$ such that $g(U) \subset \mathring{D}(\gamma_n)$.
    We first define a parametrised homotopy $h \colon X \times [0,1] \to \twThom(n)$ such that $h_0 = g$, and $h_1(U) \subset \mathring{D}(\gamma_n)$ as well as $h_1(x) = p(x)$ for $x \notin U$.
    To this end, notice that the parametrised subspace $A \coloneq \twThom(n) - z(\Gr_{n,\infty}^\infty)$ is parametrised contractible, i.e.\ there is a parametrised deformation retraction $F \colon A \times [0,1] \to A$ from $A$ onto $\RR P^\infty$ given by
    \[
        F_t(y) = \begin{cases}
                \frac{y}{\left\| y \right\|}t + (1-t)y & \text{if $y \in \mathring{D}(\gamma_n) - z(\Gr_{n,\infty}^\infty)$,} \\
                y & \text{if $y \in \RR P^\infty.$}
            \end{cases}
    \]
    Set $\lambda(x)\coloneq \left\|\psi(x)\right\|$ and define
    \[
        h_t(x) = \begin{cases}g(x) & \text{if $x \in M_g$,} \\ F_{t\lambda(x)}(g(x)) & \text{if $x \in U-M_g$,} \\ F_t(g(x)) & \text{if $x \notin U$.}\end{cases}
    \]
    Next, we wish to define a parametrised homotopy $h' \colon X \times [0,1] \to \twThom(n)$ between $h_1$ and~$f_{M_g}$.
    For this purpose, consider the vector bundle $E \coloneq (\pi_{\gamma_n} \circ \restr{g}{U})^* \gamma_n$ over $U$.
    This gives rise to a bundle map $G \colon \mathring{D}(E) \to \mathring{D}(\gamma_n)$ covering $\pi_{\gamma_n} \circ \restr{g}{U}$, and the differential of $g$ induces an isomorphism between $\nu M_g$ and $\restr{E}{M_g}$, which we will conveniently denote by $\mathrm{d}g$.
    Similar to the construction of the collapse map, choose a bundle map $\rho$ given by
    \[
    \begin{tikzcd}
        \mathring{D}(E) \arrow[r] \arrow[d, "\pi_E"'] \arrow[rr, "\rho"', bend left=30] & \mathring{D}\bigl(\restr{E}{M_g}\bigr) \arrow[r, "\left(\mathrm{d}g\right)^{-1}"] \arrow[d] & \mathring{D}(\nu M_g) \arrow[ld] \\
        U \arrow[r, "\sim", "\pi_{\nu M_g} \circ \psi"']  & M_g   & 
    \end{tikzcd}
    \]
    and identify $\overline{\psi} \colon \mathring{D}(E) \xrightarrow{\cong} \mathring{D}(\nu M_g) \times_{M_g} \mathring{D}(\nu M_g)$ as in~\eqref{eq:identify_diagonal_pullback}.
    Denote by
    \[
        \sigma \coloneq \overline{\psi}^{-1} \circ \Delta \circ \psi \colon U \to D(E)
    \]
    the section of $\pi_E$ on $U$ induced by the diagonal map $\Delta \colon \mathring{D}(\nu M_g) \to \mathring{D}(\nu M_g) \times_{M_g} \mathring{D}(\nu M_g)$.
    Then the desired parametrised homotopy $h' \colon X \times [0,1] \to \twThom(n)$ is given as follows:
    \begin{align*}
        h'_t(x) = \begin{cases}
            t\, G(\sigma(x)) + (1-t)\, h_1(x) & \text{if $x \in U$,} \\
            p(x) &  \text{if $x \notin U$.}
        \end{cases}
    \end{align*}
    Here, note that $G(\sigma(x))$ and $h_1(x)$ lie in the same fibre over $(\pi_{\gamma_n} \circ g)(x)$ for $x \in U$.
    In conclusion, $g$ is parametrised homotopic to $f_{M_g}$ as desired.
  
    Second, it remains to show that the composition $d \circ c$ is the identity on ${\LL_k(X; \Or_X \otimes \Or_\xi)}$.
    It suffices to check that $f_M^{-1}(z) = M$ and $\varphi_{f_M} = \varphi$ for a given $\xi$-oriented submanifold~$(M, \varphi)$.
    But this follows immediately from the construction, so we leave the details to the reader.
\end{proof}

\subsection{Realisability condition}\label{subsec:realisability}

Similar to the classical case, there is a realisation map
\begin{equation}\label{eq:realisation_map}
    \mu \colon \LL_k(X;\Or_X \otimes \Or_\xi) \to H_k(X;\Or_X \otimes \Or_\xi)
\end{equation}
given as follows.
Let $(M, \varphi)$ be a $\xi$-oriented submanifold of $X$ and $\iota \colon M \hookrightarrow X$ be its inclusion.
By Lemma~\ref{lemma:prop_loc_sys} there are canonical isomorphisms
\begin{align*}
    \iota^*\left( \Or_X \otimes \Or_\xi \right) &\cong \iota^*\Or_X \otimes \iota^* \Or_\xi \\
        &\cong \Or_M \otimes \Or_{\nu M} \otimes \Or_{\iota^*\xi}.
\end{align*}
Additionally, the framing $\varphi \colon \bwedge^n\!\nu M \to \iota^*\xi$ induces an isomorphism from $\Or_{\nu M} \otimes \Or_{\iota^*\xi}$ to the trivial coefficient system.
Altogether, we obtain an isomorphism $\iota^*(\Or_X \otimes \Or_\xi) \cong \Or_M$, which yields a well-defined homomorphism $\iota_* \colon H_k(M; \Or_M) \to H_k(X; \Or_X \otimes \Or_\xi)$.
In particular, the twisted fundamental class $[M]_\tw$ -- i.e.\ the canonical generator of $H_k(M; \Or_M)$ -- is mapped to $H_k(X; \Or_X \otimes \Or_\xi)$, and its image depends only on the cobordism class.
Combining the above yields a well-defined map $\mu([M, \varphi]) = \iota_*[M]_\tw$.
Lastly, we remark that the definition of the realisation map in \eqref{eq:realisation_map} justifies our choice of notation $\LL_k(X;\Or_X \otimes \Or_\xi)$ for the set of \(k\)-dimensional \(\xi\)-oriented cobordism classes (as promised at the end of Definition \ref{defn:twisted_L_equivalence}).

The preceding lemma is a crucial ingredient in our derivation of the realisability condition.

\begin{lemma}\label{lem:commutativity_realisability}
    Let $X$ be a closed $(n+k)$-dimensional manifold, $p \colon X \to \RR P^\infty$ be a (smooth) map, and $\xi = p^* \gamma_1$.
    Then following diagram commutes
    \begin{equation*}
        \begin{tikzcd}
            {[X, \twThom(n)]_{\RR P^\infty}} \arrow[d, "(\placeholder)^*(u_n^\tw)"'] \arrow[r, "d", "\cong"'] & \LL_k(X;\Or_X \otimes \Or_\xi) \arrow[d, "\mu"]           \\
            H^n(X; \mathcal{O}_\xi) \arrow[r, "\cong"' ,"{\cap[X]_\tw}"]       & H_k(X;\mathcal{O}_X \otimes \mathcal{O}_\xi) \mathrlap{,}
        \end{tikzcd}
    \end{equation*}
    where $u_n^\tw\in H^n(\twThom(n),\RR P^\infty;\mathcal{L}_\ZZ^{\mathrm{Th}})$ is the universal twisted Thom class defined in \textup{Subsection~\ref{subsec:twisted_thom_space_class}}.
\end{lemma}
\begin{proof}
    Let $f \colon X\to \twThom(n)$ be a parametrised map over $\RR P^\infty$ representing an element $[f]_{\RR P^\infty}$ in the top-left corner of the diagram, which we may also regard as a parametrised map of pairs $f\colon (X,\varnothing)\to (\twThom(n),\RR P^\infty)$. 
    The composition $\mu\circ d$ sends $[f]_{\RR P^\infty}$ into $\iota_*([M_f])=\iota_*(1\cap[M_f])$, where $\iota \colon M_f\hookrightarrow X$ is the twisted Pontryagin manifold of $f$.
    By Lemma \ref{lem:propUmkehr}, it therefore suffices to show that $f^*(u_n^\tw)=\iota^!(1)$.
    
    To start, we may assume, as in Subsection~\ref{subsec:Pontryagin}, that $f$ maps smoothly to $\mathring{D}(\gamma^L_{n,N})$ and intersects $\Gr_{n,N}^{L} \subset \mathring{D}(\gamma^L_{n,N})$ transversely. Moreover, we can choose a closed tubular neighbourhood $T$ of $\iota\colon M_f\hookrightarrow X$ in such a way that $f$ induces a map of pairs
    \[
        \restr{f}{T}\colon  (T, \partial T)\to \bigl(D(\gamma^L_{n,N}),S(\gamma^L_{n,N})\bigr) .
    \]
    Then, by applying Proposition \ref{prop:pushpull} to the transverse pullback square
    \begin{equation}
    \begin{tikzcd}
        M_f\arrow[r,"\restr{f}{M_f}"] \arrow[d,"\iota_T"', hook] \arrow[rd, phantom, "\lrcorner", very near start] & \Gr_{n,N}^{L} \arrow[d,"z", hook] \\
        T \arrow[r,"\restr{f}{T}"] &  
        D(\gamma^L_{n,N}), 
    \end{tikzcd}
    \end{equation}
     we deduce that the top square the following diagram commutes:
    \begin{equation}\label{eq:MUX}
        \begin{tikzcd}[row sep=1.25em]
            H^0(M_f;\ZZ) \arrow[d,"\iota_T^!"'] & H^0(\Gr_{n,N}^{L};\ZZ) \arrow[l,"\restr{f}{M_f}^*"] \arrow[d,"z^!"] \\
            H^n\bigl(T,\partial T;\restr{f}{T}^*\pi_{\gamma^L_{n,N}}^*\Or_{\gamma^L_{n,N}}\bigr) \arrow[d,"\cong"'] & H^n\bigl(D(\gamma^L_{n,N}),S(\gamma^L_{n,N});\pi_{\gamma^L_{n,N}}^*\Or_{\gamma^L_{n,N}}\bigr) \arrow[l,"\restr{f}{T}^*"] \arrow[d,"\cong"] \\
            H^n(T,\partial T;\restr{\Or_\xi}{T}) & H^n(D(\gamma^L_{n,N}),S(\gamma^L_{n,N});r^*\mathcal{L}_\ZZ^{\mathrm{Th}}) \arrow[l,"\restr{f}{T}^*"] \\
            H^n(X,X-\mathring{T};\Or_\xi) \arrow[u,"\cong"] \arrow[d] & H^n(\twThom(n),\RR P^\infty;\mathcal{L}_\ZZ^{\mathrm{Th}}).
            \arrow[u,"r^*"'] \arrow[l,"f^*"] \arrow[dl,"f^*", bend left=12]\\
            H^n(X;\Or_\xi) & 
        \end{tikzcd}
    \end{equation}
     Here the map $r:(D(\gamma^L_{n,N}),S(\gamma^L_{n,N}))\to (\twThom(n),\RR P^\infty)$ is the restriction of the map $r$ from Subsection \ref{subsec:twisted_thom_space_class}, where we have chosen $L$ and $N$ to be large enough to ensure that \(r\) induces an isomorphism on $n$-th degree cohomology. Note that we have made use of the coefficient isomorphisms $f^*\mathcal{L}_\ZZ^{\mathrm{Th}}\cong\Or_\xi$ and $r^*\mathcal{L}_\ZZ^{\mathrm{Th}}\cong \pi_{\gamma^L_{n,N}}^*\Or_{\gamma^L_{n,N}}$.
     
     Finally, since the left-hand vertical composition equals $\iota^!$, Lemma \ref{lem:ThomUmkehr} implies that \[z^!(1)=r^*(u_n^\tw)\in H^n(D(\gamma^L_{n,N}),S(\gamma^L_{n,N});r^*\mathcal{L}_\ZZ^{\mathrm{Th}}),\] where $1\in H^0(\Gr_{n,N}^{L};\ZZ)$ is the unit element. 
     Chasing this unit element around diagram \eqref{eq:MUX} computes $\iota^!(1)=f^*(u_n^\tw)$, as claimed.
\end{proof}

\begin{theorem}\label{thm:representability}
    Let $X^{n+k}$ be a closed $(n+k)$-dimensional manifold with $p \colon X \to \RR P^\infty$ smooth.
    Let $\AA$ be the local system of integer coefficients induced by $p$, that is $\AA = p^* \mathcal{L}_\ZZ$, where $\mathcal{L}_\ZZ$ denotes the universal coefficient system on $\RR P^\infty$ in \textup{Definition~\ref{def:universal_coefficient_system}}.
    Then, a cohomology class $\alpha \in H^n(X;\AA)$ is realisable if and only if there is a parametrised map $f \colon X\to \twThom(n)$ over $\RR P^\infty$ such that $f^*(u_n^\tw)=\alpha$.
\end{theorem}
\begin{proof}
    The statement follows immediately from Lemma~\ref{lem:commutativity_realisability}.
\end{proof}

\section{Postnikov decomposition of the twisted Thom space}
\label{sec:postnikov_decomposition}

\subsection{Generalities on Postnikov towers}
\label{subsec:general_postnikov}

All spaces are assumed to be compactly generated weak Hausdorff, and all maps are assumed to be continuous.
Let $X\in \mathcal{K}^*_G$ be a pointed $G$-space.
Recall that a space $X$ is called \emph{$G$-connected} if the fixed point set $X^H$ is path-connected for every subgroup $H\le G$.
We refer to such an $X$ as \emph{$G$-simple} if $\pi_1(X^H)$ acts trivially on $\pi_i(X^H)$ for all $H\le G$ and all $i\ge 1$ (in particular, $\pi_1(X^H)$ is abelian).
According to~\cite[Theorem II.1.2]{May96}, a $G$-simple space $X$ admits an equivariant Postnikov tower. Roughly speaking, this is a diagram of spaces and maps in $\mathcal{K}^*_G$, which, on taking $H$-fixed points, gives a Postnikov tower for the fixed point space $X^H$.
What is important for us here is that upon taking $H=\{e\}$ to be the trivial subgroup of $G$, we obtain a Postnikov tower for $X$ in the category $\mathcal{K}^*_G$.

We now make this more explicit.
Recall that we impose the ``naive'' (or Borel) model structure on $\mathcal{K}_G^*$ as described in Subsection~\ref{subsec:quillen_pair_parametrised_Gspaces}.
So, whenever we talk about fibrations, $n$-equivalences or weak equivalences, we mean it in the sense of $\mathcal{K}^*$.
Given a $G$-module $\pi$ and integer $n \ge 1$, we define a \emph{$G$-Eilenberg--MacLane space of type $(\pi,n)$} to be a connected based $G$-CW complex $K(\pi,n)$ such that $\pi_n(K(\pi,n))\cong \pi$ as $G$-modules and all other homotopy groups are trivial.
Such spaces exist \cite[Example V.4.1]{May96}. 
Note that if $X\in \mathcal{K}^*_G$, then each homotopy group $\pi_i(X)$ is a $G$-module.

\begin{definition}\label{defn:G-postnikov-tower}
    Let $X \in \mathcal{K}^*_G$ be a pointed $G$-space.
    By a \emph{Postnikov tower for $X$ in $\mathcal{K}^*_G$} we mean a diagram in $\mathcal{K}^*_G$ consisting of maps $q_n \colon X\to P[n]$ and fibrations $p_{n+1}\colon P[n+1]\to P[n]$ for $n\ge1$,  satisfying:
    \begin{enumerate}[(i)]
        \item $p_{n+1}q_{n+1}=q_n$;
        \item $q_n$ is an $(n+1)$-equivalence, that is the induced map
            \[
                (q_n)_* \colon \pi_i(X) \to \pi_i(P[n])
            \]
            is an isomorphism for $i \le n$ and an epimorphism for $i=n+1$;
        \item \(\pi_i(P[n]) = 0\), for all \(i > n\);
        \item $p_{n+1} \colon P[n+1]\to P[n]$ is induced from the based path fibration
        \[
            P K(\pi_{n+1}(X),n+2)\to K(\pi_{n+1}(X),n+2)
        \]
        via a based $G$-map $\kappa_n \colon P[n]\to K(\pi_{n+1}(X),n+2)$. 
    \end{enumerate}
\end{definition}

Recall that if $n\ge2$ and $K(\pi,n)$ is a $G$-Eilenberg--MacLane space for the module $\pi$, then
\[
    K(\pi,n) \sslash G = L_G(\pi,n)
\]
is a generalised Eilenberg--MacLane space in the sense of Subsection~\ref{subsec:genEMspaces}.
Furthermore, applying homotopy quotients to the based path fibration $K(\pi,n-1)\to P K(\pi,n)\to K(\pi,n)$ gives a fibration $K(\pi,n-1)\to \overline{P} L(\pi,n)\to L(\pi,n)$ in $\mathcal{K}^*_{/\B G}$. Here we have abbreviated \(L(\pi,n) = L_G(\pi,n)\) as established in Notation~\ref{notate:L_A}. 
These spaces and fibrations are the building blocks for Postnikov theory in the non-simple case, as described for example in~ \cite{McClendon71, Robinson72, Bausum75, Pollina82}.
In fact we have the following result:
\begin{prop}\label{prop:HQPostnikov}
    Let $X$ be a $G$-simple based $G$-space.
    Applying homotopy quotients to a Postnikov tower for $X$ in $\mathcal{K}^*_G$ as described in \textup{Subsection~\ref{subsec:quillen_pair_parametrised_Gspaces}} gives a Postnikov tower of retractive spaces for $X\sslash G \in \mathcal{K}^*_{/\B G}$. 
\end{prop}
\begin{proof}
 This follows immediately on applying Lemma \ref{lem:propsHQ}\,(iii) and (iv).  
\end{proof} 

\subsection{Postnikov tower for the twisted Thom space}

We now apply the theory of the previous Subsection~\ref{subsec:general_postnikov} to describe the beginning of a Postnikov tower for $\twThom(n)$.
For our purposes, it is convenient to identify
\[
    \twThom(n) = \MSO(n) \sslash \ZZ_2
\]
as explained in Remark~\ref{rmk:model_bo_homotopy_orbits}.

The first non-trivial homotopy group of $\MSO(n)$ is $\pi_n \coloneq \pi_n(\MSO(n))\cong \ZZ$, as follows easily from the Thom Isomorphism and Hurewicz theorems.
The $\ZZ_2$-action on $\pi_n$ induced by the action on $\MSO(n)$ described in Remark~\ref{rmk:model_bo_homotopy_orbits} is non-trivial: this follows from naturality of the Hurewicz map and an examination of the monodromy action of the relative fibration
\[
\begin{tikzcd}[sep=small]
    (D(\tilde{\gamma}_n),S(\tilde{\gamma}_n)) \arrow[r] & (D(\gamma_n),S(\gamma_n)) \arrow[r] & B\ZZ_2.
\end{tikzcd}
\]
Combined with the well known facts that $\MSO(1)\simeq K(\ZZ,1)$ and $\MSO(2)\simeq K(\ZZ,2)$, this gives the following in low dimensions:
\begin{lemma}\label{lem:k12}
   For $n=1,2$, the Thom space $\MSO(n)$ is a $\ZZ_2$-Eilenberg--MacLane space $K(\ZZ,n)$, where $\ZZ_2$ acts on $\ZZ$ by negation.
   Consequently, in these dimensions $\twThom(n)$ is a generalised Eilenberg--MacLane space $L(\ZZ,n)$.  
\end{lemma}

From now on we assume $n\ge3$.
Recall that Thom has shown in ~\cite[\S II.8, \S II.9]{Thom54} that the next non-trivial homotopy group of $\MSO(n)$ after $\pi_n\cong \ZZ$ is $\pi_{n+4} \coloneq \pi_{n+4}(\MSO(n))\cong \ZZ$, and furthermore that the beginning of the Postnikov tower for $\MSO(n)$ looks as follows:
\begin{equation}\label{eq:GPtower}
    \begin{tikzcd}
       && {\vdots} & \\
    	&& {P[n+4]} & \\
    	 {\MSO(n)} && {K(\pi_n,n)} & {K(\pi_{n+4},n+5)} \mathrlap{.}
        \arrow["{\St^5_3}", from=3-3, to=3-4]
     	\arrow["{p_{n+4}}",from=2-3, to=3-3]
    	\arrow["{q_{n+4}}" description, from=3-1, to=2-3, bend left=20]
    	\arrow["{q_n}", from=3-1, to=3-3]
    	  \arrow[from=1-3, to=2-3]
    \end{tikzcd}
\end{equation}
Here the $(n+1)$-equivalence $q_n$ may be chosen to represent either of the Thom class generators of  $H^n(\MSO(n);\pi_n)\cong \ZZ$. The first $\kappa$-invariant represents the class
\[
    \St^5_3(\iota)\in H^{n+5}(K(\ZZ,n);\ZZ)\cong H^{n+5}(K(\pi_n,n);\pi_{n+4}),
\]
where $\St^5_3=\beta \, P^1_3 \, \rho_3$ is the Steenrod cube (see Subsection~\ref{subsec:classical_computations}).
Since $\MSO(n)$ is a based $\ZZ_2$-simple space (as $\MSO(n)$ and $\MSO(n)^{\ZZ_2}=\{\ast\}$ are simply connected), it admits a Postnikov tower in $\mathcal{K}^*_{\ZZ_2}$ in the sense of Subsection~\ref{subsec:general_postnikov}.
This will have the exact same form as the tower~\eqref{eq:GPtower}, but now all the maps are based $\ZZ_2$-maps and the Eilenberg--MacLane spaces are $\ZZ_2$-Eilenberg--MacLane spaces for the $\ZZ_2$-modules $\pi_n$ and $\pi_{n+4}$.
In particular, by the uniqueness of the first $\kappa$-invariant, we find an equivariant representative in the homotopy class of
\[
    \St^5_3 \colon K(\pi_n,n)\longrightarrow K(\pi_{n+4},n+5).
\]
Note that by the fact that $\St^5_3(-\iota)=-\St^5_3(\iota)$, this implies that the $\ZZ_2$-action on $\pi_{n+4}$ is also the non-trivial action.

Now taking homotopy quotients in diagram~\eqref{eq:GPtower} and applying Proposition~\ref{prop:HQPostnikov}, we obtain a Postnikov tower for $\twThom(n)$ in $\mathcal{K}^*_{/\B\ZZ_2}$:
\begin{equation}\label{eq:Ptower}
\begin{tikzcd}
  && {\vdots} & \\
  && {P[n+4]\sslash\ZZ_2} & \\
  \twThom(n) && {L(\ZZ,n)} & {L(\ZZ,n+5)} \mathrlap{.}
  \arrow["{\St^5_3 \sslash \ZZ_2}", from=3-3, to=3-4]
  \arrow["{p_{n+4}\sslash\ZZ_2}", from=2-3, to=3-3]
  \arrow["{q_{n+4}\sslash\ZZ_2}" description, from=3-1, to=2-3, bend left=20]
  \arrow["{q_n\sslash\ZZ_2}", from=3-1, to=3-3]
  \arrow[from=1-3, to=2-3]
\end{tikzcd}
\end{equation}

\begin{lemma}\label{lem:postnikov_tower}
    The following two statements hold:
    \begin{enumerate}[label=\upshape{(\roman*)}]
        \item The map $q_n\sslash\ZZ_2 \colon \twThom(n) \to L(\ZZ,n)$ may be chosen to represent the twisted Thom class
        \[
            u_n^\tw\in H^n(\twThom(n),\B\ZZ_2; \mathcal{L}_\ZZ^\mathrm{Th}). 
        \]
        \item The map
        \[
            \St^5_{3, \tw} \coloneq \St^5_3\sslash\ZZ_2 \colon L(\ZZ,n)\to L(\ZZ,n+5)
        \]
        represents the twisted cohomology operation 
        \[
            \St^5_{3, \tw} =\beta^\mathrm{tw}\circ \mathcal{P}^1_3\circ\rho_3^{\mathrm{tw}}.
        \]
        Here $\mathcal{P}^1_3$ is Gitler's twisted cohomology operation lifting the Steenrod reduced power $P^1_3:H^*(\placeholder;\ZZ_3)\to H^{*+4}(\placeholder;\ZZ_3)$, and ${\rho_3}^\mathrm{tw}$ and $\beta^{\mathrm{tw}}$ are respectively the twisted reduction mod $3$ and the twisted Bockstein associated with the exact coefficient sequence
        \[
        \begin{tikzcd}[sep=small]
           0 \arrow[r] & \AA \arrow[r,"\times 3"] & \AA \arrow[r,"{\rho_3^{\mathrm{tw}}}"] & \AA_3\arrow[r] & 0.
        \end{tikzcd}
        \]
        \end{enumerate}
\end{lemma}
\begin{proof}
\leavevmode
    \begin{enumerate}[label=\upshape{(\roman*)}]
        \item We have a commuting diagram
        \[
        \begin{tikzcd}[sep=1.5em, nodes={font=\small}, labels={font=\scriptsize}]
            (D(\tilde\gamma_n)_x,S(\tilde\gamma_n)_x) \arrow[r, hook] \arrow[d,"\jmath"'] &  (D(\tilde\gamma_n),S(\tilde\gamma_n)) \arrow[r,"q_n"] \arrow[d] & (K(\ZZ,n),*) \arrow[d] \\
             (D(\tilde\gamma_n\sslash\ZZ_2)_{[x,e]},S(\tilde\gamma_n\sslash\ZZ_2)_{[x,e]}) \arrow[d] \arrow[r, hook] & (D(\tilde\gamma_n\sslash\ZZ_2),S(\tilde\gamma_n\sslash\ZZ_2)) \arrow[r,"q_n\sslash \ZZ_2"] \arrow[d] & (L(\ZZ,n),\B\ZZ_2) \arrow[d]\\
             \{[e]\} \arrow[r, hook] & \B\ZZ_2 \arrow[r,equal] & \B\ZZ_2
        \end{tikzcd}
        \]
        in which $x\in \BSO(n)$ and $e\in \E\ZZ_2$ are arbitrary.
        Since $q_n$ is a $(n+1)$-equivalence, it represents a Thom class in $\widetilde{H}^n(\MSO(n);\ZZ)\cong H^n(D(\tilde\gamma_n),S(\tilde\gamma_n);\ZZ)$.
        The twisted fundamental class $e_n\in H^n(L(\ZZ,n),\B\ZZ_2;\mathcal{L}_\ZZ^n)$ restricts to the usual fundamental class $\iota_n\in \widetilde{H}^n(K(\ZZ,n);\ZZ)$, and so continuing to pull it back to the top left corner gives a generator of $H^n(D(\tilde\gamma_n)_x,S(\tilde\gamma_n)_x;\ZZ)$.
        On the other hand, the map $\jmath$ induces an isomorphism \[
            H^n \bigl(D(\tilde\gamma_n\sslash\ZZ_2)_{[x,e]},S(\tilde\gamma_n\sslash\ZZ_2)_{[x,e]};\ZZ \bigr)\cong H^n \bigl(D(\tilde\gamma_n)_x,S(\tilde\gamma_n)_x;\ZZ \bigr),
        \]
        so that pulling back $e_n$ along the composition in the middle row gives a generator also.
        \item Each of the untwisted operations \(\theta \in \{\rho_3, P^1_3, \beta\}\) satisfies $\theta(-x)=-\theta(x)$ whereby there are $\ZZ_2$-equivariant representatives between $\ZZ_2$-Eilenberg--MacLane spaces
        \begin{align*}
            \rho_3 &\colon K(\ZZ,n)\to K(\ZZ_3,n), \\
            P^1_3 &\colon K(\ZZ_3,n)\to K(\ZZ_3,n+4), \text{ and }\\
            \beta &\colon K(\ZZ_3,n+4)\to K(\ZZ,n+5).
        \end{align*}
        Now applying homotopy orbits we get representatives of the twisted operations
        \begin{align*}
            \rho_3^{\mathrm{tw}} &\colon L(\ZZ,n)\to L(\ZZ_3,n), \\
            \mathcal{P}^1_3 &\colon L(\ZZ_3,n)\to L(\ZZ_3,n+4), \text{ and } \\
            {\beta}^{\mathrm{tw}} &\colon L(\ZZ_3,n+4)\to L(\ZZ,n+5),
        \end{align*}
        compare the proof of Theorem 8.5 in \cite[pp.~184--185]{Gitler63}. The functoriality of homotopy orbits now implies the result.
    \end{enumerate}
This completes the proof. 
\end{proof}

\begin{theorem}\label{thm:twistedcube}
    Assume that $X$ is a smooth manifold and let $\AA$ be the local integer coefficient system induced by some map $p \colon X \to \B\ZZ_2$.
    Let $\alpha \in H^n(X;\AA)$ be a twisted cohomology class.
    \begin{enumerate}[label=\upshape{(\roman*)}]
        \item If $n=1,2$ then $\alpha$ is realisable.
        \item If $n \ge 3$ then a necessary condition for $\alpha$ to be realisable is that $\St^5_{3, \tw} (\alpha)\in H^{n+5}(X;\AA)$ is zero.
        If $\dim(X)\le n+5$, this condition is also sufficient. 
    \end{enumerate}
\end{theorem}
\begin{proof}
    Note that $(X,p)$ is a parametrised space over $\B\ZZ_2$.
    Since we want to work with retractive spaces, we add a disjoint parametrised base point by defining $X_+ \coloneq X \sqcup B\ZZ_2$ and extend $p$ to $X_+$ by the identity map on $\B\ZZ_2$.
    Then for any $Y \in \mathcal{K}^*_{/\B\ZZ_2}$ we have
    \[
        [(X_+,\B\ZZ_2),Y]^\mathrm{rel}_{\B\ZZ_2} \cong [X,Y]_{\B\ZZ_2}.
    \]
    Then, by virtue of Theorem~\ref{thm:representability}, the class $\alpha\in H^n(X;\AA)\cong H^n(X_+,\B\ZZ_2;\AA)$ is realisable if and only if there is a retractive map $f \colon X_+\to \twThom(n)$ such that $f^*(u_n^\tw)=\alpha$, where $u_n^\tw\in H^n(M,\B\ZZ_2;\mathcal{L}_\ZZ^\mathrm{Th})$ is the twisted Thom class.
    \begin{enumerate}[(i)]
        \item This follows from Lemma~\ref{lem:k12}.
        \item Let $\alpha \colon X_+\to L(\ZZ,n)$ be a parametrised map representing $\alpha$.
            The composition ${\St^5_{3, \tw} \circ (q_n\sslash\ZZ_2)}$ is retractively null-homotopic, so if $\alpha$ lifts through the Thom class
            \[
                u_n^\tw\simeq_{\B\ZZ_2} q_n\sslash\ZZ_2 \colon \twThom(n) \longrightarrow L(\ZZ,n)
            \]
            then $\St^5_{3, \tw} \circ\alpha$ is null-homotopic, so $\St^5_{3, \tw} (\alpha)\in H^{n+5}(X;\AA)$ is zero.
            Conversely, if ${\St^5_{3, \tw} (\alpha)\in H^{n+5}(X;\AA)}$ is zero then $\alpha$ lifts through
            \[
                p_{n+4}\sslash\ZZ_2 \colon P[k+4]\sslash\ZZ_2 \longrightarrow L(\ZZ,n).
            \]
            The retractive map $q_{n+4}\sslash\ZZ_2$ between fibrant retractive spaces induces isomorphisms on homotopy groups up to degree $n+4$ and an epimorphism in degree $n+5$.
            Thus if $X$ has dimension at most $n+5$ then by Proposition~\ref{prop:n-equiv}, the map $\alpha$ lifts through 
            \[
                (q_{n+4}\sslash\ZZ_2) \circ (p_{n+4}\sslash\ZZ_2) = q_n\sslash\ZZ_2 \simeq_{\B\ZZ_2} u_n^\tw. 
            \]
    \end{enumerate}
    This completes the proof.
\end{proof}

\begin{remark}
    Using known calculations of the stable groups $\pi_{n+i}(\MSO(n))$ for $i<n$, which agree with the oriented bordism groups $\Omega_i$, this can be pushed a little further.
    For example, just using Thom's calculations (see~\cite[Theorems II.16, II.17]{Thom54}) we can say that if $n\ge8$ and $\dim(X)\le n+8$ then $\alpha\in H^n(X;\AA)$ is realisable if and only if $\St^5_{3, \tw} (\alpha)=0$.
\end{remark}

\section{An application: Realisability in non-orientable manifolds}\label{sec:application}

Thom's results on realisability of integral homology classes assume that the ambient manifold is orientable, so that the Poincar\'e dual cohomology class is integral.
The question of whether a given integral homology class $a\in H_k(X;\ZZ)$ in a non-orientable manifold $X^{n+k}$ can be realised by an oriented submanifold was raised recently by Zhenhua Liu on MathOverflow~\cite{LiuMO}, and can be addressed using the methods in this paper.
In particular, such an $a$ is realisable if and only if the dual cohomology class $\alpha\in H^n(X;\Or_X)$ is an image of the twisted Thom class.
Applying Theorem~\ref{thm:twistedcube} we immediately obtain the following.

\begin{corollary}
    Let $X$ be a non-orientable $(n+k)$-manifold. 
    \begin{enumerate}[label=\upshape{(\roman*)}]
        \item Every homology class $a\in H_{k}(X;\ZZ)$ is realisable by an orientable submanifold if $n=1$ or $n=2$.
        \item If $n\ge 3$ and $k\le 5$ then $a\in H_{k}(X;\ZZ)$ is realisable if and only if the dual cohomology class $\alpha\in H^n(X;\Or_X)$ satisfies $\St^5_{3, \tw} (\alpha)=0$. 
    \end{enumerate}
\end{corollary}

We now give examples of non-realisable twisted integral cohomology classes in non-orientable manifolds.
Our method is to find non-realisable integral cohomology classes in orientable manifolds $\widetilde{X}$ which are anti-invariant under some orientation-reversing free involution, since such classes often come from twisted integral cohomology classes on the quotient $X$, and we may apply the following lemma.

\begin{lemma}\label{lem:liftrep}
    Let $X$ be an $(n+k)$-dimensional manifold and $\AA$ be the coefficient system induced by some map $p \colon X \to \B\ZZ_2$.
    Suppose $\alpha\in H^n(X;\AA)$ is realisable, and let $\rho\colon\widetilde{X}\to X$ be the double cover corresponding to the map $p \colon X \to \B\ZZ_2$.
    Then $\rho^*(\alpha)\in H^n(\widetilde{X};\ZZ)$ is realisable.
\end{lemma}
\begin{proof}
    Let $\iota\colon M^{k}\hookrightarrow X^{n+k}$ be an embedding realising $\alpha$.
    According to the definition of the Umkehr homomorphism and Lemma~\ref{lem:propUmkehr}, this is equivalent to $\iota^!(1)=\alpha$, where $\iota\in H^0(M;\ZZ)\cong H^0(M;\iota^*\AA\otimes\Or_{\nu\iota})$ is the unit element.
    Now we consider the transverse pullback square
    \[
    \begin{tikzcd}
        \widetilde{M} \arrow[r,"\rho_M"] \arrow[d,"\widetilde{\iota}"', hook] \arrow[rd, phantom, "\lrcorner", very near start] & M  \arrow[d,"\iota", hook] \\
        \widetilde{X} \arrow[r,"\rho"] & X
    \end{tikzcd}
    \]
    and apply Proposition~\ref{prop:pushpull} to the element $1\in H^0(M;\ZZ)$.
    Here, the map $\rho_M \colon \widetilde{M} \to M$ is the double cover of $M$ induced from $\rho$.
    This gives
    \[
        \rho^*(\alpha)=\rho^*\iota^!(1)=\widetilde{\iota}^! \rho_M^*(1)=\widetilde{\iota}^!(1).
    \]
    Thus the embedding $\widetilde{\iota} \colon \widetilde{M}\hookrightarrow\widetilde{X}$ realises $\rho^*(\alpha)$.
\end{proof}

\begin{example}\label{ex:maptoruslens}
    We now give an example of a non-realisable $7$-dimensional integral homology class in a non-orientable $11$-manifold.
    We leverage Bohr--Hanke--Kotschick's example~\cite{BohrHankeKotschick00} of a non-realisable class in a $10$-dimensional product of lens spaces.
    Let $L_1=S^7/\ZZ_3$ and $L_2=S^3/\ZZ_3$ be standard lens spaces (that is, quotients of the unit sphere in some $\mathbb{C}^n$ by the action of $\ZZ_3=\langle \omega\rangle\subseteq \UU(1)$ given by $\omega\cdot (z_1,\ldots, z_n)=(\omega z_1,\ldots, \omega z_n)$, where $\omega$ is a cube root of unity). For $i=1,2$ let $v_i\in H^1(L_i;\ZZ_3)\cong\ZZ_3$ denote a generator.
    Consider the class $u=\beta(v_1\times v_2)\in H^3(L_1\times L_2;\ZZ)$.
    It is shown in \cite{BohrHankeKotschick00} that
    \[
        0\neq \St^5_3(u)\in H^8(L_1\times L_2;\ZZ),
    \]
    hence $u$ and its Poincar\'e dual homology class in $H_7(L_1\times L_2;\ZZ)$ are non-realisable.
    
    On each $L_i$ there is a free involution induced by the antipodal map on the covering sphere.
    This is orientation-preserving on the sphere (as a product of an even number of reflections) and hence also on $L_i$.
    Now consider the product manifold $\widetilde{X} \coloneq L_1\times L_2\times S^1$, with the diagonal involution $\tau$ formed from the antipodal actions on the lens space factors and complex conjugation on the circle factor.
    This is a free, orientation-reversing involution, so the quotient $X \coloneq \widetilde{X}/\tau$ is a non-orientable manifold.
    
    Now consider the cohomology class
    \[
        \widetilde{\alpha} \coloneq u\times\sigma\in H^4(\widetilde{X};\ZZ),
    \]
    where $\sigma\in H^1(S^1;\ZZ)$ denotes a generator.
    Using the Cartan formula, the product formula for Bocksteins, and the K\"unneth formula, one easily checks that
    \[
        0\neq  \St^5_3(\widetilde{\alpha})= \St^5_3(u\times \sigma) =  \St^5_3(u)\times\sigma \in H^9(\widetilde{X};\ZZ),
    \]
    hence $\widetilde{\alpha}$ is non-realisable.
    
    Next we observe that $\tau^*(\widetilde{\alpha})=-\widetilde{\alpha}$, that is, $\widetilde{\alpha}$ is anti-invariant under the involution $\tau$.
    Firstly, note that complex conjugation sends $\sigma$ to $-\sigma$, so it suffices to check that the diagonal involution on $L_1\times L_2$ induced by the antipodal map on each lens space factor sends $u=\beta(v_1\times v_2)$ to itself.
    By additivity of the Bockstein, it suffices to check that the antipodal map on each $L_i$ sends $v_i\in H^1(L_i;\ZZ_3)$ to itself.
    To see this, note that the inclusion $S^1/\ZZ_3\hookrightarrow L_i$ induces the epimorphism $\ZZ\twoheadrightarrow\ZZ_3$ on fundamental groups, hence induces an injection $H^1(L_i;\ZZ_3)\to H^1(S^1/\ZZ_3;\ZZ_3)$, and on $S^1/\ZZ_3 \cong S^1$ the antipodal map has degree one.    
    
    Finally, let $\rho \colon \widetilde{X}\to X$ be the (orientation) double cover, with induced map
    \[
        \rho^* \colon H^*(X;\Or_X)\longrightarrow H^*(\widetilde{X};\ZZ)
    \]
    and associated transfer
    \[
        \rho^! \colon H^*(\widetilde{X};\ZZ)\longrightarrow H^*(X;\Or_X).
    \]
    Using the transfer formula (see e.g. \cite[III.9.5\,(iii)]{Brown82}) we have 
    \[
        \rho^*\rho^!(\widetilde{\alpha}) = \widetilde{\alpha} - \tau^*(\widetilde{\alpha}) = 2\widetilde{\alpha} = -\widetilde{\alpha},
    \]
    since $3\widetilde{\alpha}=0$. Finally, set $\alpha\coloneq\rho^!(-\widetilde{\alpha})\in H^4(X;\Or_X)$.
    This class is non-realisable by Lemma~\ref{lem:liftrep}, since $\widetilde{\alpha}=\rho^*(\alpha)$ is non-realisable.
    The Poincar\'e dual of $\alpha$ in $H_7(X;\ZZ)$ is the non-realisable integral homology class we seek.
    
    Using the results of Gitler \cite{Gitler63} which relate twisted cohomology operations with equivariant cohomology operations in the associated covering space, we can observe that in this example we have
    \[
        0\neq \St^5_{3, \tw} (\alpha)\in H^9(X;\Or_X).
    \]
\end{example}

\begin{remark}
    In the published version~\cite{BohrHankeKotschick02} of~\cite{BohrHankeKotschick00}, Bohr--Hanke--Kotschick give a different example of a non-realisable $7$-dimensional homology class in a $10$-manifold, namely the Poincar\'e dual of a generator $u\in H^3(\operatorname{Sp}(2);\ZZ)$.
    In both examples, we were unable to find an orientation-reversing free involution of the $10$-manifold for which the class $u$ is anti-invariant, thus preventing us from using this technique to find an example of minimal dimension.
    Furthermore the operation used to detect non-realisability in the $\operatorname{Sp}(2)$ example, namely $\rho_3(\placeholder)\cup P^1_3\circ\rho_3(\placeholder)$, is non-stable, and vanishes on $u\times \sigma\in H^4(\operatorname{Sp}(2)\times S^1;\ZZ)$.
    We suspect this class is realisable. 
\end{remark}

\begin{example}\label{ex:dim_10_degree_3}
    The following example gives a non-realisable $7$-dimensional integral homology class in a non-orientable $10$-manifold.
    Let $Y \coloneq L_1\times L_2$ be the product of lens spaces from Example \ref{ex:maptoruslens}, and let $P \coloneq \RR P^{10}$ be the non-orientable real projective space of dimension $10$.
    The connected sum
    \[
        X \coloneq Y\# P
    \]
    is non-orientable.
    Denote by $Y'$ and $P'$ the manifolds $Y$ and $P$ with an open disc removed, and denote by $\Or_Y'$ and $\Or_P'$ the restrictions of $\Or_Y$ and $\Or_P$ to $Y'$ and $P'$, respectively.
    Using the exact sequences in cohomology associated with the cofibre sequences $S^9\to Y'\to Y$ and $S^9\to P'\to P$, we see that the restrictions induce isomorphisms $H^i(Y;\Or_Y)\cong H^i(Y';\Or_Y')$ and $H^i(P;\Or_P)\cong H^i(P';\Or_P')$, for $i\le8$.
    Observe that the restrictions of $\Or_X$ to $Y'$ and $P'$ respectively agree with $\Or_Y'$ and $\Or_P'$.
    Thus the Mayer--Vietoris sequence for the union $X=Y'\cup_{S^9} P'$ gives isomorphisms
    \[
        H^i(X;\Or_X)\cong H^i(Y';\Or_Y')\oplus H^i(P';\Or_P'), \qquad \text{for $i\le8$}. 
    \]
    Putting things together, we have for each $i\le8$ a diagram
    \[
    \begin{tikzcd}
       H^i(X;\Or_X) \arrow[r,"\cong"] & H^i(Y';\Or_Y')\oplus H^i(P';\Or_P') & \arrow[l,"\cong" swap] H^i(Y;\Or_Y)\oplus H^i(P;\Or_P)
    \end{tikzcd}
    \]
    where the arrows are inclusion-induced isomorphisms.
    Since $Y$ is orientable, we have that  $H^i(Y;\Or_Y)=H^i(Y;\ZZ)$.
    Consider the element $U\in H^3(X;\Or_X)$ corresponding to
    \[
        (u,0)\in H^3(Y;\ZZ)\oplus H^3(P;\Or_P)
    \]
    under the above isomorphisms, where $u=\beta(v_1\times v_2)$ is the element considered in Example~\ref{ex:maptoruslens}.
    By naturality of twisted cohomology operations, the element $\St^5_{3, \tw} (U)\in H^8(X;\Or_X)$ corresponds to the nonzero element 
    \[
        \St^5_{3, \tw} (u,0)=(\St^5_3(u),0)\in H^8(Y;\Or_Y)\oplus H^8(P;\Or_P),
    \]
    and is therefore nonzero.
    (Here we have used additivity of twisted cohomology operations on a disjoint union, and the fact that they agree with the corresponding untwisted operations for trivial coefficient systems.)
    Hence by Theorem \ref{thm:twistedcube}\,(ii), the class $U$ and its Poincar\'e dual in $H_7(X;\ZZ)$ are non-realisable.
\end{example}
\begin{remark}
    Zhenhua Liu \cite{LiuMO} asks what is the smallest pair of natural numbers $(n,k)$ such that there exists a non-realisable integral homology class of dimension $k$ in a non-orientable manifold of dimension $n+k$. Using the results of this paper we can say that $(n,k)=(3,6)$ or $(3,7)$. Thom claims \cite[Footnote, p.~56]{Thom54} that $(3,6)$ cannot occur in the oriented (untwisted) case. It seems unlikely that any proof of this claim would easily generalise to the twisted case.
\end{remark}

Due to our inability to transfer Thom's claim \cite[Footnote, p.~56]{Thom54} to the parametrised setup, we cannot validate whether Example~\ref{ex:dim_10_degree_3} is of minimal dimension.
Even though we expect our example to be minimal, the question remains:
\begin{question}
     Does there exist a non-realisable integral homology class of dimension $3$ in a non-orientable manifold of dimension $9$?
\end{question}

Moreover, Zhenhua Liu \cite{LiuMO} also asks, in the spirit of Thom's result~\cite[Theorem II.29]{Thom54}, whether a multiple of some integer homology class in a non-orientable manifold $X$ can be represented by an orientable submanifold of $X$.
Even though the present article provides the necessary foundations to tackle this question, it does not answer it.
This leads us to formulate the following more general question for arbitrary integer coefficient systems.

\begin{question}\label{ques:multiple_homology}
    For any homology class $z \in H_k(X; \AA)$ of a manifold $X$ with local integer coefficients $\AA$, does there exist an integer $N \neq 0$ such that the class $Nz$ is realisable by a submanifold?
\end{question}

A generalisation to arbitrary local coefficient systems is desirable but expected to be difficult.
Primarily, it is not clear to us what \emph{realisability} for coefficient systems with structure groups other than $G = \ZZ_2$ could mean since, in such a situation, there is no obvious twisted fundamental class to choose from. 

\section*{Acknowledgements}

We thank Oliver Wang for his suggestion to consider the connected sum with $\RR P^{10}$ to produce Example~\ref{ex:dim_10_degree_3}.
We thank Thomas Rot for comments on an earlier draft, and for valuable discussions with the first two authors in Amsterdam. The first author thanks Diarmuid Crowley, Irakli Patchkoria and Andr\'as Sz\H{u}cs for helpful remarks.
Moreover, we thank Cary Malkiewich for his valuable comments on the first version of this manuscript.

\section*{Generative AI disclosure}

The second author used Microsoft Copilot (GPT-5.5) to discuss the details of Lemma~\ref{lemma:prop_loc_sys}, which inspired the current form of the proof.
Otherwise, all mathematical results, proofs and conclusions were developed and verified by the authors of this article, who take full responsibility for its content.

\printbibliography

\end{document}